\documentclass[preprint,12pt]{elsarticle}

\usepackage{amsmath,amsfonts,amssymb,amscd,amsthm,amsbsy,upref}
\usepackage{amsthm}

\numberwithin{equation}{section}
\newtheorem{thm}{Theorem}[section]
\newtheorem{lem}[thm]{Lemma}
\newtheorem{cor}[thm]{Corollary}

\theoremstyle{definition}
\newtheorem{defn}{Definition}[section]

\journal{}

\begin{document}

\begin{frontmatter}



\title{Pe{\l}czy\'{n}ski's property ($V^{*}$) of order $p$ and its quantification\tnoteref{label1}}
\tnotetext[label1]{Lei Li's research is partly supported by the NSF of China (11301285).
Dongyang Chen's project was supported by the Natural Science Foundation of Fujian Province of China (No. 2015J01026).
}


\author{Lei Li}
\address{School of Mathematical Sciences and LPMC, Nankai University, Tianjin, 300071, China}
\ead{leilee@nankai.edu.cn}

\author{Dongyang Chen\fnref{label2}}
\address{School of Mathematical Sciences, Xiamen University,
Xiamen,361005,China}
\ead{cdy@xmu.edu.cn}
\fntext[label2]{Corresponding author}

\author{J. Alejandro Ch\'{a}vez-Dom\'{i}nguez}
\address{Department of Mathematics, University of Oklahoma, Norman, Oklahoma, 73019,USA}
\ead{jachavezd@math.ou.edu}

\begin{abstract}
We introduce the concepts of Pe{\l}czy\'{n}ski's property ($V$) of order $p$ and Pe{\l}czy\'{n}ski's property ($V^{*}$) of order $p$. It is proved that, for each $1<p<\infty$, the James $p$-space $J_{p}$ enjoys Pe{\l}czy\'{n}ski's property ($V^{*}$) of order $p$ and the James $p^{*}$-space $J_{p^{*}}$ (where $p^{*}$ denotes the conjugate number of $p$) enjoys Pe{\l}czy\'{n}ski's property ($V$) of order $p$. We prove that
both $L_{1}(\mu)$ ($\mu$ a finite positive measure) and $l_{1}$ enjoy the quantitative version of Pe{\l}czy\'{n}ski's property ($V^{*}$).
\end{abstract}

\begin{keyword}
Pe{\l}czy\'{n}ski's property ($V$) of order $p$ \sep Pe{\l}czy\'{n}ski's property ($V^{*}$) of order $p$ \sep Quantifying Pe{\l}czy\'{n}ski's property ($V^*$) of order $p$


\MSC 46B20

\end{keyword}

\end{frontmatter}


\section{Introduction and notations}
Let $X$ and $Y$ be Banach spaces. Recall that an operator $T:X\rightarrow Y$ is called \textit{unconditionally converging} if $T$ takes weakly unconditionally Cauchy series in $X$ to unconditionally converging series in $Y$. In his fundamental paper \cite{P}, A. Pe{\l}czy\'{n}ski introduced property ($V$). A Banach space $X$ is said to have \textit{Pe{\l}czy\'{n}ski's property ($V$)} if every unconditionally converging operator with domain $X$ is unconditionally converging. Equivalently, $X$ has Pe{\l}czy\'{n}ski's property ($V$) if a bounded subset $K$ of $X^{*}$ is relatively weakly compact whenever $\lim_{n\rightarrow \infty}\sup_{x^{*}\in K}|<x^{*},x_{n}>|=0$ for every weakly unconditionally Cauchy series $\sum_{n=1}^{\infty}x_{n}$ in $X$. The most known classical Banach spaces that have Pe{\l}czy\'{n}ski's property ($V$) are spaces $C(\Omega)$ of continuous scalar-valued functions on compact Hausdorff space $\Omega$ \cite{P}, or more generally Banach spaces whose duals are isometric to $L_{1}$-spaces \cite{JZ}. Pe{\l}czy\'{n}ski's property ($V^{*}$) was introduced in \cite{P} as a dual property of Pe{\l}czy\'{n}ski's property ($V$). A Banach space $X$ is said to have \textit{Pe{\l}czy\'{n}ski's property ($V^{*}$)} if a bounded subset $K$ of $X$ is relatively weakly compact whenever $\lim_{n\rightarrow \infty}\sup_{x\in K}|<x^{*}_{n},x>|=0$ for every weakly unconditionally Cauchy series $\sum_{n=1}^{\infty}x^{*}_{n}$ in $X^{*}$. Among classical Banach spaces that have Pe{\l}czy\'{n}ski's property ($V^{*}$), $L_{1}$-spaces are the most notable ones.

The main goal of this paper is to generalize Pe{\l}czy\'{n}ski's property ($V$) and Pe{\l}czy\'{n}ski's property ($V^{*}$) to more general case. In Section 2, we introduce the concept of Pe{\l}czy\'{n}ski's property ($V$) of order $p$ $(1\leq p\leq\infty)$ (property $p$-($V$) in short). Property $1$-($V$) is precisely Pe{\l}czy\'{n}ski's property ($V$) and property $\infty$-($V$) is precisely the reciprocal Dunford-Pettis property (see \cite{KS} for this definition). It is clear that for each $1<p<\infty$, a Banach space $X$ has property $p$-($V$) whenever $X$ has  Pe{\l}czy\'{n}ski's property ($V$). It is natural to ask whether there exists a space enjoying property $p$-($V$) fails Pe{\l}czy\'{n}ski's property ($V$) for each $1<p<\infty$. In this section, we show that for each $1<p<\infty$, the James $p^{*}$-space $J_{p^{*}}$ (where $p^{*}$ denotes the conjugate number of $p$) has property $p$-($V$) (see Theorem \ref{3.9} below). But, $J_{p^{*}}$ clearly fails Pe{\l}czy\'{n}ski's property ($V$) because $J_{p^{*}}$ contains no copy of $c_{0}$ and is non-reflexive. It is proved in \cite{CG} that Pe{\l}czy\'{n}ski's property ($V$) is not a three-space property, that is, there exist a space $X$ failing Pe{\l}czy\'{n}ski's property ($V$) and a closed subspace $X_{0}$ of $X$ such that both $X_{0}$ and the quotient $X/X_{0}$ have Pe{\l}czy\'{n}ski's property ($V$). We extend this result to property $p$-($V$) and show that property $p$-($V$) are not three-space properties for each $1\leq p<\infty$. The concept of Pe{\l}czy\'{n}ski's property ($V^{*}$) of order $p$ (property $p$-($V^{*}$) in short) is introduced in this section. Property $1$-($V^{*}$) is precisely Pe{\l}czy\'{n}ski's property ($V^{*}$). Similarly, for each $1<p<\infty$, a Banach space $X$ has property $p$-($V^{*}$) whenever $X$ has Pe{\l}czy\'{n}ski's property ($V^{*}$). We show that the converse is false. For each $1<p<\infty$, the James $p$-space $J_{p}$ fails Pe{\l}czy\'{n}ski's property ($V^{*}$) because $J_{p}$ is not weakly sequentially complete. But $J_{p}$ enjoys property $p$-($V^{*}$) for each $1<p<\infty$ as shown in the following Theorem \ref{3.22}. In \cite{P}, A. Pe{\l}czy\'{n}ski proved that if a Banach space $X$ has both property ($V$) and property ($V^{*}$), then $X$ must be reflexive. However, Theorem \ref{3.9} and Theorem \ref{3.22} in this section tell us that the classical non-reflexive James space $J$ has both property $2$-($V$) and property $2$-($V^{*}$). A. Pe{\l}czy\'{n}ski showed in \cite{P} that if a Banach space $X$ has property ($V^{*}$), then $X$ must be weakly sequentially complete. Correspondingly, we introduce the notion of weak sequential completeness of order $p$ and show that if a Banach space $X$ has property $p$-($V^{*}$), then $X$ must be weakly sequentially complete of order $p$ for each $1<p<2$.

In \cite{B}, F. Bombal studied Pe{\l}czy\'{n}ski's property ($V^{*}$) in vector-valued sequence spaces and proved that given a sequence $(X_{n})_{n}$ of Banach spaces, the space $(\sum_{n=1}^{\infty}\oplus X_{n})_{p}$($1\leq p<\infty$) has Pe{\l}czy\'{n}ski's property ($V^{*}$) if and only if each $X_{n}$ does. Our Theorem \ref{3.13} and Theorem \ref{3.201} in Section 3 cover this result.
Moreover, we characterize $p$-($V$) sets and prove that the space $(\sum_{n=1}^{\infty}\oplus X_{n})_{p}$($1<p<\infty$) has property $q$-($V$)($1\leq q<\infty$) if and only if each $X_{n}$ does. In particular, we show that the space $(\sum_{n=1}^{\infty}\oplus X_{n})_{p}$($1<p<\infty$ or $p=0$)
has Pe{\l}czy\'{n}ski's property ($V$) if and only if each $X_{n}$ does.

Section 3 is concerned with quantifications of property $p$-($V^{*}$) and property $p$-($V$). H. Kruli\v{s}ov\'{a} \cite{K} introduced several possibilities of quantifying Pe{\l}czy\'{n}ski's property ($V$) and proved a quantitative version of Pe{\l}czy\'{n}ski's result about $C(K)$ spaces. More precisely, he proved that the space $C_{0}(\Omega)$ enjoys the quantitative property $(V_{q})^{*}_{\omega}$ with constant $\pi$ ($2$ in the real case) for every locally compact Hausdorff space $\Omega$. In this section, we introduce the concepts of quantitative Pe{\l}czy\'{n}ski's property $(V^{*})$ of order $p$ and quantitative Pe{\l}czy\'{n}ski's property $(V)$ of order $p$. First we prove quantitative versions of some results about
property $p$-($V^{*}$) and property $p$-($V$). It is proved in \cite{KKS} that the quantities $\omega(\cdot)$ and $wk(\cdot)$ are equal in $L_{1}(\mu)$ for a general positive measure $\mu$. In this section, we introduce a quantity $\iota_{p}(\cdot)(1\leq p<\infty)$ and prove that the quantities $wk(\cdot)$ and $\iota_{1}$ are equal in $L_{1}(\mu)$ ($\mu$ a finite positive measure) and $l_{1}$. In particular, both $L_{1}(\mu)$ ($\mu$ a finite positive measure) and $l_{1}$ have quantitative Pe{\l}czy\'{n}ski's property $(V^{*})$ with constant $1$. Finally, we show that
$c_{0}$ enjoys the quantitative property $(V_{q})^{*}_{\omega}$ with constant $1$.

Our notation and terminology are standard as may be found in \cite{AK} and \cite{LT}. Throughout the paper, all Banach spaces can be considered either real or complex unless stated otherwise. By an operator, we always mean a bounded linear operator.
$p^{*}$ will always denote the conjugate number of $p$ for $1\leq p<\infty$.
Let $X$ be a Banach space, $1\leq p<\infty$ and we denote $l^{w}_{p}(X)$ by the space of all weakly $p$-summable sequences in $X$, endowed with the norm $$\|(x_{n})_{n}\|_{p}^{w}=\sup\{(\sum_{n=1}^{\infty}|<x^{*},x_{n}>|^{p})^{\frac{1}{p}}:x^{*}\in B_{X^{*}}\}, \quad (x_{n})_{n}\in l^{w}_{p}(X).$$
A sequence $(x_{n})_{n}\in l^{w}_{p}(X)$ is \textit{unconditionally $p$-summable} if $$\sup\{(\sum_{n=m}^{\infty}|<x^{*},x_{n}>|^{p})^{\frac{1}{p}}:x^{*}\in B_{X^{*}}\}\rightarrow 0\quad\text{as}\; m\rightarrow \infty.$$
In \cite{CCL}, we extend unconditionally converging operators and completely continuous operators to the general case $1\leq p\leq \infty$. Let $1\leq p\leq \infty$. We say that an operator $T:X\rightarrow Y$ is \textit{unconditionally $p$-converging} if $T$ takes weakly $p$-summable sequences (weakly null sequences for $p=\infty$) to unconditionally $p$-summable sequences (norm null sequences for $p=\infty$).


\section{Pe{\l}czy\'{n}ski's property ($V$) of order $p$ and Pe{\l}czy\'{n}ski's property ($V^{*}$) of order $p$}

\begin{defn}
Let $1\leq p\leq \infty$. We say that a Banach space $X$ has \textit{Pe{\l}czy\'{n}ski's property ($V$) of order $p$} (property $p$-($V$) in short) if for every Banach space $Y$, every unconditionally $p$-converging operator $T:X\rightarrow Y$ is weakly compact.
\end{defn}

Obviously, for every $1\leq p<q\leq \infty$, a Banach space $X$ has property $q$-($V$) whenever $X$ has property $p$-($V$).

\begin{defn}\cite{CCL}
Let $X$ be a Banach space and $1\leq p\leq \infty$. We say that a bounded subset $K$ of $X^{*}$ is \textit{a $p$-($V$) set} if $$\lim_{n\rightarrow \infty}\sup_{x^{*}\in K}|<x^{*},x_{n}>|=0,$$ for every $(x_{n})_{n}\in l^{w}_{p}(X)$ ($(x_{n})_{n}\in c^{w}_{0}(X)$ for $p=\infty$).
\end{defn}
$1$-($V$) sets are called ($V$)-sets in \cite{CE} and $\infty$-($V$) sets are called ($L$)-sets (see \cite{E1} for example).
Before giving a useful characterization of a $p$-($V$) set, we recall the notion of weakly $p$-convergent sequences introduced in \cite{CS}. Let $1\leq p\leq \infty$. A sequence $(x_{n})_{n}$ in a Banach space $X$ is said to be \textit{weakly $p$-convergent to $x\in X$} if the sequence $(x_{n}-x)_{n}$ is weakly $p$-summable in $X$. Weakly $\infty$-convergent sequences are simply the weakly convergent sequences. The concept of weakly $p$-Cauchy sequences is introduced in \cite{CCL}. We say that a sequence $(x_{n})_{n}$ in a Banach space $X$ is \textit{weakly $p$-Cauchy} if for each pair of strictly increasing sequences $(k_{n})_{n}$ and $(j_{n})_{n}$ of positive integers, the sequence $(x_{k_{n}}-x_{j_{n}})_{n}$ is weakly $p$-summable in $X$.
Obviously, every weakly $p$-convergent sequence is weakly $p$-Cauchy, and the weakly $\infty$-Cauchy sequences are precisely the weakly Cauchy sequences.
J.M.F.Castillo and F.S\'{a}nchez said that a Banach space $X\in W_{p}(1\leq p\leq\infty)$ if any bounded sequence in $X$ admits a weakly $p$-convergent subsequence (see \cite{CS}). The following characterization of a $p$-($V$) set appears in \cite{CCL}.

\begin{thm}\cite{CCL}\label{3.8}
Let $1<p<\infty$ and $X$ be a Banach space. The following statements are equivalent about a bounded subset $K$ of $X^{*}$:
\item[(1)]$K$ is a $p$-($V$) set;
\item[(2)]For all spaces $Y\in W_{p}$ and for every operator $T$ from $Y$ into $X$, the subset $T^{*}(K)$ is relatively norm compact;
\item[(3)]For every operator $T$ from $l_{p^{*}}$ into $X$, the subset $T^{*}(K)$ is relatively norm compact.
\end{thm}

In case of $p=1$, R. Cilia and G. Emmanuele proved that a bounded subset $K$ of $X^{*}$ is a $1$-($V$) set if and only if for every operator $T$ from $c_{0}$ into $X$, the subset $T^{*}(K)$ is relatively norm compact (see \cite{CE}). However, the equivalence between (1) and (3) of Theorem \ref{3.8} is false for $p=\infty(p^{*}=1)$. For instance, by Schur property, every bounded subset of $l_{\infty}$ is a $\infty$-($V$) set, but $B_{l_{\infty}}$ is not relatively norm compact.

Let us fix some notations. If $A$ and $B$ are nonempty subsets of a Banach space $X$, we set $$d(A,B)=\inf\{\|a-b\|:a\in A,b\in B\},$$$$\widehat{d}(A,B)=\sup\{d(a,B):a\in A\}.$$ Thus, $d(A,B)$ is the ordinary distance between $A$ and $B$, and $\widehat{d}(A,B)$ is the non-symmetrized Hausdorff distance from $A$ to $B$.

Let $X$ be a Banach space and $A$ be a bounded subset of $X^{*}$. For $1\leq p\leq\infty$, we set
\begin{center}
$\xi_{p}(A)=\inf\{\widehat{d}(A,K): K\subset X^{*}$ is a $p$-($V$) set $\}.$
\end{center}
It is clear that $\xi_{p}(A)=0$ if and only if $A$ is a $p$-($V$) set.

Let $A$ be a bounded subset of a Banach space $X$.
The de Blasi measure of weak non-compactness of $A$ is defined by
\begin{center}
$\omega(A)=\inf\{\widehat{d}(A,K):\emptyset\neq K\subset X$ is weakly compact $\}.$
\end{center}
Then $\omega(A)=0$ if and only if $A$ is relatively weakly compact.
For an operator $T: X\rightarrow Y$, we denote $\xi_{p}(TB_{X}), \omega(TB_{X})$ by $\xi_{p}(T), \omega(T)$ respectively.

To characterize property $p$-($V$), we need the following result in \cite{CCL}.

\begin{thm}\cite{CCL}\label{3.2}
Let $1\leq p\leq \infty$. The following statements about an operator $T:X\rightarrow Y$ are equivalent:
\item[(1)]$T$ is unconditionally $p$-converging;
\item[(2)]$T$ sends weakly $p$-summable sequences onto norm null sequences;
\item[(3)]$T$ sends weakly $p$-Cauchy sequences onto norm convergent sequences.
\end{thm}

The following result shows that Pe{\l}czy\'{n}ski's property ($V$) of order $p$ is automatically quantitative in some sense.
\begin{thm}\label{3.1}
Let $1\leq p\leq\infty$ and $X$ be a Banach space. The following are equivalent:
\item[(1)]$X$ has property $p$-($V$);
\item[(2)]Every $p$-(V) subset of $X^{*}$ is relatively weakly compact;
\item[(3)]$\omega(T^{*})\leq \xi_{p}(T^{*})$ for every operator $T$ from $X$ into any Banach space $Y$;
\item[(4)]$\omega(A)\leq \xi_{p}(A)$ for every bounded subset $A$ of $X^{*}$.
\end{thm}
\begin{proof}
$(1)\Rightarrow(2)$. Suppose that $K$ is a $p$-($V$) subset of $X^{*}$. Take any sequence $(x^{*}_{n})_{n}$ in $K$. Define an operator $T: X\rightarrow l_{\infty}$ by $$Tx=(<x^{*}_{n},x>)_{n}, \quad x\in X.$$ Then, for every $(x_{n})_{n}\in l^{w}_{p}(X)$, we have $$\|Tx_{n}\|=\sup_{k}|<x^{*}_{k},x_{n}>|\leq \sup_{x^{*}\in K}|<x^{*},x_{n}>|\rightarrow 0 \quad (n\rightarrow \infty).$$ It follows from Theorem \ref{3.2} that $T$ is unconditionally $p$-converging. By (1), the operator $T$ is weakly compact and hence $T^{*}$ is also weakly compact. This implies that the set $T^{*}B_{l^{*}_{\infty}}$ is relatively weakly compact. It is easy to see that $T^{*}e_{n}=x^{*}_{n}$ for each $n\in \mathbb{N}$, where $(e_{n})_{n}$ is the unit vector basis of $l_{1}$. So the sequence $(x^{*}_{n})_{n}$ is relatively weakly compact.

$(2)\Rightarrow(3)$ is obvious. $(3)\Rightarrow(1)$ is immediate from Theorem \ref{3.2}. The equivalence of $(2)\Leftrightarrow (4)$ is straightforward.
\end{proof}

\begin{cor}\label{3.6}
Let $1\leq p\leq \infty$. If a Banach space $X$ has property $p$-($V$), then every quotient of $X$ has property $p$-($V$).
\end{cor}

Recall that the James $p$-space $J_{p}(1<p<\infty)$ is the (real) Banach space of all sequences $(a_{n})_{n}$ of real numbers such that $\lim_{n\rightarrow \infty}a_{n}=0$ and
\begin{eqnarray*}
\|(a_{n})_{n}\|_{cpv}&=&\frac{1}{2^{\frac{1}{p}}}\sup\{(\sum_{j=1}^{m}|a_{i_{j-1}}-a_{i_{j}}|^{p}+|a_{i_{m}}-a_{i_{0}}|^{p})^{\frac{1}{p}}:1\leq i_{0}<i_{1}<\cdots<i_{m}, m\in \mathbb{N}\}\\
&<&\infty.
\end{eqnarray*}
Another useful equivalent norm on $J_{p}$ is given by the formula
$$\|(a_{n})_{n}\|_{pv}=\sup\{(\sum_{j=1}^{m}|a_{i_{j-1}}-a_{i_{j}}|^{p})^{\frac{1}{p}}:1\leq i_{0}<i_{1}<\cdots<i_{m}, m\in \mathbb{N}\}.$$
In fact,
\begin{equation}\label{69}
\frac{1}{2^{\frac{1}{p}}}\|\cdot\|_{pv}\leq \|\cdot\|_{cpv}\leq \|\cdot\|_{pv}.
\end{equation}
The sequence $(e_{n})_{n}$ of standard unit vectors forms a monotone shrinking basis for $J_{p}$ in both norms $\|\cdot\|_{pv}$ and $\|\cdot\|_{cpv}$. It is known that $J_{p}$ is non-reflexive and is codimension of $1$ in $J^{**}_{p}$, but every infinite-dimensional closed subspace of $J_{p}$ contains a subspace isomorphic to $l_{p}$.

The following lemma may appear somewhere. Its proof is identical to \cite[Proposition 3.4.3]{AK}.
\begin{lem}\label{3.90}
Let $(x_{k})_{k}$ be a normalized block basic sequence with respect to $(e_{n})_{n}$ in $(J_{p},\|\cdot\|_{pv})$. Then, for any sequence $(\lambda_{k})_{k=1}^{n}$ of real numbers and any $n\in\mathbb{N}$ the following estimate holds:
$$\|\sum_{k=1}^{n}\lambda_{k}x_{k}\|_{pv}\leq (1+2^{p})^{\frac{1}{p}}(\sum_{k=1}^{n}|\lambda_{k}|^{p})^{\frac{1}{p}}.$$
\end{lem}

\begin{thm}\label{3.9}
The James $p$-space $J_{p}$ has property $p^{*}$-($V$).
\end{thm}
\begin{proof}
Let $K$ be a $p^{*}$-($V$) subset of $B_{J_{p}^{*}}$. Take any sequence $(x^{*}_{n})_{n}$ from $K$. Since $J_{p}$ is separable, we may assume that $(x^{*}_{n})_{n}$ is $weak^{*}$-convergent to some $x^{*}_{0}\in J^{*}_{p}$.

Claim. $(x^{*}_{n})_{n}$ converges to $x^{*}_{0}$ weakly.

Note that $J^{**}_{p}=J_{p}\oplus span\{x^{**}_{0}\}$, where $x^{**}_{0}\in J^{**}_{p}$ is defined by $<x^{**}_{0},e^{*}_{n}>=1$ for all $n\in \mathbb{N}$, where $(e^{*}_{n})_{n}$ is the functionals biorthogonal to the unit vector basis $(e_{n})_{n}$ of $J_{p}$. Thus it suffices to prove that
$$<x^{**}_{0},x^{*}_{n}>\rightarrow <x^{**}_{0},x^{*}_{0}> \quad (n\rightarrow \infty).$$
Suppose it is false. By passing to subsequences, we may assume that $|<x^{**}_{0},x^{*}_{n}-x^{*}_{0}>|>\epsilon_{0}$ for some $\epsilon_{0}>0$ and for all $n\in \mathbb{N}$. Since $(e_{n})_{n}$ is shrinking, $(e^{*}_{n})_{n}$ forms a basis for $J^{*}_{p}$. Thus
\begin{equation}\label{63}
|\sum_{k=1}^{\infty}<x^{*}_{n}-x^{*}_{0},e_{k}>|=|<x^{**}_{0},\sum_{k=1}^{\infty}<x^{*}_{n}-x^{*}_{0},e_{k}>e^{*}_{k}>|>\epsilon_{0},\quad n=1,2,\cdots
\end{equation}
Note that
\begin{equation}\label{64}
\lim_{n\rightarrow \infty}<x^{*}_{n}-x^{*}_{0},e_{k}>=0, \quad k=1,2,\cdots
\end{equation}
By inductions on $n$ in (\ref{63}) and on $k$ in (\ref{64}), we obtain $1=n_{1}<n_{2}<n_{3}<\cdots$ and $0=k_{0}<k_{1}<k_{2}<k_{3}<\cdots$ such that
\begin{equation}\label{65}
|\sum_{k=k_{j-1}+1}^{k_{j}}<x^{*}_{n_{j}}-x^{*}_{0},e_{k}>|>\frac{\epsilon_{0}}{2},\quad j=1,2,\cdots
\end{equation}
We set $x_{j}=\sum_{k=k_{j-1}+1}^{k_{j}}e_{k}(j=1,2,\cdots)$. Then $(x_{j})_{j}$ is a normalized block basic sequence with respect to $(e_{n})_{n}$ in $(J_{p}, \|\cdot\|_{cpv})$. It follows from (\ref{69}) and Lemma \ref{3.90} that for any sequence of real scalars $(\lambda_{j})_{j=1}^{n}$ the following estimate holds
$$\|\sum_{j=1}^{n}\lambda_{j}x_{j}\|_{cpv}\leq (2+2^{p+1})^{\frac{1}{p}}(\sum_{j=1}^{n}|\lambda_{j}|^{p})^{\frac{1}{p}}$$
This implies that $(x_{j})_{j}$ is weakly $p^{*}$-summable. Since $K$ is a $p^{*}$-($V$) set, we get
$$|<x^{*}_{n_{j}},x_{j}>|\leq \sup_{x^{*}\in K}|<x^{*},x_{j}>|\rightarrow 0 \quad(j\rightarrow \infty).$$
Obviously, $(<x^{*}_{0},x_{j}>)_{j}$ converges to $0$. Therefore, we have
$$|<x^{*}_{n_{j}}-x^{*}_{0},x_{j}>|\rightarrow 0 \quad(j\rightarrow \infty),$$
which contradicts with (\ref{65}). This contradiction shows that $(x^{*}_{n})_{n}$ converges to $x^{*}_{0}$ weakly. Thus $K$ is relatively weakly compact. By Theorem \ref{3.1}, $J_{p}$ has property $p^{*}$-($V$).

\end{proof}

As in \cite{CG}, we consider the space $X_{p}$ constructed in \cite{FGJ}. We do not describe the space $X_{p}$ here and refer the reader to \cite{FGJ} for details. In \cite{FGJ}, a quotient map $T_{p}: X_{p}\rightarrow c_{0}$ is defined and it is proved that $T_{p}$ is unconditionally converging. We extend this result as follows:

\begin{lem}\label{3.200}
For $1<p<\infty$, the quotient map $T_{p}$ is unconditionally $q$-converging for any $1\leq q<p^{*}$.
\end{lem}
\begin{proof}
Suppose that $T_{p}$ is not unconditionally $q$-converging for some $1\leq q<p^{*}$. Then there exists an operator $S$ from  $l_{q^{*}}$ ($c_{0}$ for $q=1$) into $X_{p}$ such that $T_{p}S$ is non-compact. Thus, we can find a weakly null sequence $(z_{n})_{n}$ in $l_{q^{*}}$ and $\epsilon_{0}>0$ such that $\|T_{p}Sz_{n}\|\geq \epsilon_{0}$ for each $n\in \mathbb{N}$.
By passing to subsequences, we may assume that $(z_{n})_{n}$ is equivalent to the unit vector basis of $l_{q^{*}}$, that is, there exist $C_{1}, C_{2}>0$ such that for all $n\in \mathbb{N}$ and all scalars $\alpha_{1},\alpha_{2},...,\alpha_{n}$, one has
\begin{equation}\label{60}
C_{1}(\sum_{k=1}^{n}|\alpha_{k}|^{q^{*}})^{\frac{1}{q^{*}}}\leq \|\sum_{k=1}^{n}\alpha_{k}z_{k}\|\leq C_{2}(\sum_{k=1}^{n}|\alpha_{k}|^{q^{*}})^{\frac{1}{q^{*}}}.
\end{equation}
Let $x_{n}=Sz_{n}$. By \cite[Proposition 2]{GJ}, the sequence $(x_{n})_{n}$ admits a subsequence, which is still denoted by
$(x_{n})_{n}$, such that $(x_{2n-1}-x_{2n})_{n}$ is equivalent to the unit vector basis of $l_{p}$. Then, there exist $D_{1}, D_{2}>0$ such that for all $n\in \mathbb{N}$ and all scalars $\alpha_{1},\alpha_{2},...,\alpha_{n}$, one has
\begin{equation}\label{61}
D_{1}(\sum_{k=1}^{n}|\alpha_{k}|^{p})^{\frac{1}{p}}\leq \|\sum_{k=1}^{n}\alpha_{k}(x_{2k-1}-x_{2k})\|\leq D_{2}(\sum_{k=1}^{n}|\alpha_{k}|^{p})^{\frac{1}{p}}.
\end{equation}
By (\ref{60}) and (\ref{61}), we get, for each $n$ and scalars $\alpha_{1},\alpha_{2},...,\alpha_{n}$,
\begin{align*}
D_{1}(\sum_{k=1}^{n}|\alpha_{k}|^{p})^{\frac{1}{p}}&\leq \|\sum_{k=1}^{n}\alpha_{k}(x_{2k-1}-x_{2k})\|\\
&\leq \|S\|\cdot \|\sum_{k=1}^{n}\alpha_{k}(z_{2k-1}-z_{2k})\|\\
&\leq \|S\|\cdot C_{2}\cdot 2^{\frac{1}{q^{*}}}(\sum_{k=1}^{n}|\alpha_{k}|^{q^{*}})^{\frac{1}{q^{*}}},
\end{align*}
which is impossible because $1\leq q<p^{*}$. This completes the proof.

\end{proof}

\begin{thm}
Property $q$-($V$) is not a three-space property for each $1\leq q<\infty$.
\end{thm}
\begin{proof}
For $1\leq q<\infty$, choose $1<p<\infty$ with $q<p^{*}$. It is shown in \cite{CG} that both $X_{p}/Ker(T_{p})$ and $Ker(T_{p})$ have property $1$-($V$) and hence have property $q$-($V$). But $X_{p}$ fails property $q$-($V$) since, by Lemma \ref{3.200}, $T_{p}$ is unconditionally $q$-converging, but obviously not weakly compact.

\end{proof}

\begin{defn}
Let $X$ be a Banach space and $1\leq p\leq \infty$. We say that a bounded subset $K$ of $X$ is \textit{a $p$-($V^{*}$) set} if $$\lim_{n\rightarrow \infty}\sup_{x\in K}|<x^{*}_{n},x>|=0,$$ for every $(x^{*}_{n})_{n}\in l^{w}_{p}(X^{*})$ ($(x^{*}_{n})_{n}\in c^{w}_{0}(X^{*})$ for $p=\infty$).
\end{defn}
It is noted that $1$-($V^{*}$) sets are ($V^{*}$)-sets (see \cite{P}) and $\infty$-($V^{*}$) sets are Dunford-Pettis sets.

\begin{thm}\label{3.4}
Let $K$ be a bounded subset of a Banach space $X$ and $1<p<\infty$. The following statements are equivalent:
\item[(1)]$K$ is a $p$-($V^{*}$) set;
\item[(2)]For all spaces $Y$ with $Y^{*}\in W_{p}$, every operator $T: X\rightarrow Y$ maps $K$ onto a relatively norm compact subset of $Y$;
\item[(3)]Every operator $T: X\rightarrow l_{p}$ maps $K$ onto a relatively norm compact subset of $l_{p}$.
\end{thm}
\begin{proof}
$(1)\Rightarrow (2)$. Let $Y$ and $T$ be as stated in (2). Assume the contrary that $T(K)$ is not relatively norm compact.
Then there exists a sequence $(x_{n})_{n}$ in $K$ such that $(Tx_{n})_{n}$ admits no norm convergent subsequences.
Since $Y$ is reflexive, by passing to a subsequence if necessary we may assume that $(Tx_{n})_{n}$ converges weakly to some $y\in Y$ and $\|Tx_{n}-y\|>\epsilon_{0}$ for some $\epsilon_{0}>0$ and for all $n\in \mathbb{N}$. For each $n\in \mathbb{N}$, choose $y^{*}_{n}$ with $\|y^{*}_{n}\|\leq 1$ such that $|<y^{*}_{n},Tx_{n}-y>|>\epsilon_{0}$. Since $Y^{*}\in W_{p}$, by passing to a subsequence again one can assume that the sequence $(y^{*}_{n})_{n}$ is weakly $p$-convergent to some $y^{*}\in Y^{*}$. By (1), we get $$\lim_{n\rightarrow \infty}\sup_{x\in K}|<T^{*}y^{*}_{n}-T^{*}y^{*},x>|=0.$$ For each $n\in \mathbb{N}$, we have
\begin{align*}
\epsilon_{0}&<|<y^{*}_{n},Tx_{n}-y>|\\
&\leq |<y^{*}_{n},Tx_{n}>-<y^{*},Tx_{n}>|+|<y^{*},Tx_{n}>-<y^{*},y>|\\
&  \qquad +|<y^{*},y>-<y^{*}_{n},y>|\\
&\leq \sup_{x\in K}|<T^{*}y^{*}_{n}-T^{*}y^{*},x>|+|<y^{*},Tx_{n}-y>|\\
&  \qquad +|<y^{*}-y^{*}_{n},y>|\rightarrow 0 \quad (n\rightarrow \infty),\\
\end{align*}
which is a contradiction.

$(2)\Rightarrow (3)$ is immediate because $l_{p^{*}}\in W_{p}$;

$(3)\Rightarrow (1)$. Let $(x^{*}_{n})_{n}\in l^{w}_{p}(X^{*})$. Then there exists an operator $T$ from $X$ into $l_{p}$ such that $Tx=(<x^{*}_{n},x>)_{n}$ for all $x\in X$. It follows from (3) that $T(K)$ is relatively norm compact. By the well-known characterization of relatively norm compact subsets of $l_{p}$, one can derive that
$\lim_{n\rightarrow \infty}\sup_{x\in K}|<x^{*}_{n},x>|=0$. This finishes the proof.

\end{proof}

It should be mentioned that G. Emmanuele proved the equivalence between (1) and (3) of Theorem \ref{3.4} for $p=1$ (see \cite{E}). Obviously, this is false for $p=\infty$, for example, take $X=c_{0}$. But, K. T. Andrews proved that a bounded subset $K$ of a Banach space $X$ is a $\infty$-($V^{*}$) set if and only if every weakly compact operator $T: X\rightarrow c_{0}$ maps $K$ onto a relatively norm compact subset (see \cite{A}).

Let $X$ be a Banach space and $A$ be a bounded subset of $X$. For $1\leq p\leq\infty$, we set
\begin{center}
$\theta_{p}(A)=\inf\{\widehat{d}(A,K): K\subset X$ is a $p$-($V^{*}$) set $\}.$
\end{center}
Obviously, $\theta_{p}(A)=0$ if and only if $A$ is a $p$-($V^{*}$) set.

\begin{defn}
Let $1\leq p\leq \infty$. We say that a Banach space $X$ has \textit{Pe{\l}czy\'{n}ski's property ($V^{*}$) of order $p$} ($p$-($V^{*}$) in short) if every $p$-($V^{*}$) subset of $X$ is relatively weakly compact.
\end{defn}
It is clear that for every $1\leq p<q\leq \infty$, a Banach space $X$ has property $q$-($V^{*}$) whenever $X$ has property $p$-($V^{*}$).

The proof of the following lemma is similar to \cite[Proposition 5]{A1}.
\begin{lem}\label{3.45}
Let $(x^{*}_{n})_{n}=(\sum_{i=k_{n-1}+1}^{k_{n}}a_{i}e^{*}_{i})_{n}$ be a semi-normalized block basic sequence with respect to $(e^{*}_{n})_{n}$ in $J^{*}_{p}$ and suppose that $\sum_{i=k_{n-1}+1}^{k_{n}}a_{i}=0$ for each $n\in \mathbb{N}$. Then $(x^{*}_{n})_{n}$ is equivalent to the unit vector basis of $l_{p^{*}}$.
\end{lem}

\begin{thm}\label{3.22}
The James $p$-space $J_{p}$ has property $p$-($V^{*}$).
\end{thm}
\begin{proof}
Let $K$ be a $p$-($V^{*}$) subset of $B_{J_{p}}$. Take any sequence $(x_{n})_{n}$ from $K$. Since $J^{*}_{p}$ is separable, we may assume that $(x_{n})_{n}$ is $weak^{*}$-convergent to some $x^{**}\in B_{J^{**}_{p}}$. It aims to prove that $x^{**}\in J_{p}$, that is, $\lim_{k\rightarrow \infty}<x^{**},e^{*}_{k}>=\xi=0$.

Suppose that $\xi\neq 0$. Let $\delta=\frac{|\xi|}{2}>0$. Then there exists $p_{1}\in \mathbb{N}$ such that
$|<x^{**},e^{*}_{k}>|>\delta$ for all $k\geq p_{1}$. Since $(x_{n})_{n}$ is $weak^{*}$-convergent to $x^{**}$, we choose $n_{1}$ such that
$|<e^{*}_{p_{1}},x_{n_{1}}>|>\delta$. Choose $q_{1}>p_{1}$ such that $|<e^{*}_{k},x_{n_{1}}>|<\frac{\delta}{2}$ for each $k\geq q_{1}$. In particular, $|<e^{*}_{q_{1}},x_{n_{1}}>|<\frac{\delta}{2}$. Choose any $p_{2}>q_{1}$. Then there exists $n_{2}>n_{1}$ such that $|<e^{*}_{p_{2}},x_{n_{2}}>|>\delta$. Choose $q_{2}>p_{2}$ such that $|<e^{*}_{k},x_{n_{2}}>|<\frac{\delta}{2}$ for each $k\geq q_{2}$. In particular, $|<e^{*}_{q_{2}},x_{n_{2}}>|<\frac{\delta}{2}$. We continue in a similar manner and obtain
\begin{center}
$p_{1}<q_{1}<p_{2}<q_{2}<\cdots$ and $n_{1}<n_{2}<\cdots$
\end{center}
such that
\begin{center}
$|<e^{*}_{p_{j}},x_{n_{j}}>|>\delta$ and $|<e^{*}_{q_{j}},x_{n_{j}}>|<\frac{\delta}{2}, \quad j=1,2,\cdots$
\end{center}
Set $z^{*}_{j}=e^{*}_{p_{j}}-e^{*}_{q_{j}}(j=1,2,\cdots)$. Then, for each $j\in \mathbb{N}$, we have
\begin{align*}
|<z^{*}_{j},x_{n_{j}}>|&=|<e^{*}_{p_{j}},x_{n_{j}}>-<e^{*}_{q_{j}},x_{n_{j}}>|\\
&\geq |<e^{*}_{p_{j}},x_{n_{j}}>|-|<e^{*}_{q_{j}},x_{n_{j}}>|\\
&>\delta-\frac{\delta}{2}=\frac{\delta}{2}.\\
\end{align*}
Thus $(z^{*}_{j})_{j}$ is a semi-normalized block basic sequence of $(e^{*}_{n})_{n}$. It follows from Lemma \ref{3.45} that $(z^{*}_{j})_{j}$ is equivalent to the unit vector basis of $l_{p^{*}}$. In particular, $(z^{*}_{j})_{j}$ is weakly $p$-summable. Since $K$ is a $p$-($V^{*}$) set, we get
$$\frac{\delta}{2}<|<z^{*}_{j},x_{n_{j}}>|\leq \sup_{x\in K}|<z^{*}_{j},x>|\rightarrow 0 \quad (j\rightarrow \infty),$$
which is a contradiction.

\end{proof}

The proof of the following theorem is similar to Theorem \ref{3.1}.

\begin{thm}
Let $1\leq p\leq\infty$ and $X$ be a Banach space. The following are equivalent:
\item[(1)]$X$ has property $p$-($V^{*}$);
\item[(2)]For all spaces $Y$, an operator $T: Y\rightarrow X$ is weakly compact whenever $T^{*}$ is unconditionally $p$-converging;
\item[(3)]$\omega(T)\leq \theta_{p}(T)$ for every operator $T$ from any Banach space $Y$ into $X$;
\item[(4)]$\omega(A)\leq \theta_{p}(A)$ for every bounded subset $A$ of $X$.
\end{thm}

\begin{cor}\label{3.10}
Let $1\leq p\leq \infty$. If a Banach space $X$ has property $p$-($V^{*}$), then every closed subspace of $X$ has property $p$-($V^{*}$).
\end{cor}

\begin{cor}\label{3.5}
Let $1\leq p\leq\infty$ and $X$ be a Banach space. Then
\item[(1)]If $X$ has property $p$-($V$), then $X^{*}$ has property $p$-($V^{*}$);
\item[(2)]If $X^{*}$ has property $p$-($V$), then $X$ has property $p$-($V^{*}$).
\end{cor}

We remark that the converse of Corollary \ref{3.5} is not true for all $1\leq p\leq\infty$. J. Bourgain and F. Delbaen (see \cite{BD}) constructed a Banach space $X_{BD}$ such that $X_{BD}$ has the Schur property, $X^{*}_{BD}$ is isomorphic to an $L_{1}$-space. Thus, the space $X_{BD}$ fails property $p$-($V$) for all $1\leq p\leq\infty$. Since $X^{*}_{BD}$ is isomorphic to an $L_{1}$-space, $X^{*}_{BD}$ has property $1$-($V^{*}$) and hence property $p$-($V^{*}$) for all $1\leq p\leq\infty$.

\begin{defn}
Let $1\leq p\leq \infty$. We say that a Banach space $X$ is \textit{weakly sequentially complete of order $p$} if every weakly $p$-Cauchy sequence in $X$ is weakly $p$-convergent.
\end{defn}

The weakly sequential completeness of order $\infty$ is precisely the classical weakly sequential completeness. It is easy to verify that for $1\leq p<q\leq \infty$, a Banach space $X$ is weakly sequentially complete of order $p$ whenever $X$ is weakly sequentially complete of order $q$.
\begin{thm}
Let $1<p<2$. If a Banach space $X$ has property $p$-($V^{*}$), then $X$ is weakly sequentially complete of order $p$.
\end{thm}
\begin{proof}
It follows from $1<p<2$ that the identity $I_{p}: l_{p}\rightarrow l_{p}$ is unconditionally $p$-converging. By Theorem \ref{3.2}, we see that every weakly $p$-Cauchy sequence in $l_{p}$ is convergent in norm. Let $(x_{n})_{n}$ be a weakly $p$-Cauchy sequence in $X$. Then, for every operator $T: X\rightarrow l_{p}$, the sequence $(Tx_{n})_{n}$ is weakly $p$-Cauchy and hence is convergent in norm. By Theorem \ref{3.4}, we get that $(x_{n})_{n}$ is a $p$-($V^{*}$) set. Since $X$ has property $p$-($V^{*}$), the sequence $(x_{n})_{n}$ is relatively weakly compact. Thus, $(x_{n})_{n}$ is weakly $p$-convergent.
\end{proof}

\begin{cor}
Let $1<p<2$. If a Banach space $X$ has property $p$-($V$), then $X^{*}$ is weakly sequentially complete of order $p$.
\end{cor}

\section{Pe{\l}czy\'{n}ski's property ($V$) of order $p$ and Pe{\l}czy\'{n}ski's property ($V^{*}$) of order $p$ in vector-valued sequence spaces}

Let $(X_{n})_{n}$ be a sequence of Banach spaces and $1\leq p<\infty$. We denote by $(\sum_{n=1}^{\infty}\oplus X_{n})_{p}$ the space of all vector-valued sequences $x=(x_{n})_{n}$ with $x_{n}\in X_{n}(n\in \mathbb{N})$, for which $$\|x\|=(\sum_{n=1}^{\infty}\|x_{n}\|^{p})^{\frac{1}{p}}<\infty.$$
Similarly, $(\sum_{n=1}^{\infty}\oplus X_{n})_{0}$ denotes the space of all vector-valued sequences $x=(x_{n})_{n}$ with $x_{n}\in X_{n}(n\in \mathbb{N})$, for which $\lim_{n\rightarrow \infty}\|x_{n}\|=0$, endowed with the supreme norm. The direct sum in the sense of $l_{\infty}$ of $(X_{n})_{n}$, denoted by $(\sum_{n=1}^{\infty}\oplus X_{n})_{\infty}$, is defined in an analogous way. For every $n\in \mathbb{N}$, $I_{n}$ will denote the canonical injection from $X_{n}$ into $(\sum_{n=1}^{\infty}\oplus X_{n})_{p}$ and $\pi_{n}$ will denote the canonical projection from $(\sum_{n=1}^{\infty}\oplus X_{n})_{p}$ into $X_{n}$. We denote the canonical injection $J_{n}$ from $X^{*}_{n}$ into $(\sum_{n=1}^{\infty}\oplus X^{*}_{n})_{p^{*}}$ and the canonical projection from $(\sum_{n=1}^{\infty}\oplus X^{*}_{n})_{p^{*}}$ onto $X^{*}_{n}$ by $P_{n}$. Clearly, $I^{*}_{n}=P_{n}$ and $\pi^{*}_{n}=J_{n}$.

\begin{thm}\label{3.11}
Let $(X_{n})_{n}$ be a sequence of Banach spaces and let $X=(\sum_{n=1}^{\infty}\oplus X_{n})_{p}(1<p<\infty)$ or $X=(\sum_{n=1}^{\infty}\oplus X_{n})_{0}$. The following are equivalent for a bounded subset $A$ of $X^{*}$:
\item[(1)]$A$ is a $p^{*}$-($V$) set;
\item[(2)]$P_{n}(A)$ is a $p^{*}$-($V$) set for each $n\in \mathbb{N}$ and $$\lim_{n\rightarrow \infty}\sup\{\sum_{k=n}^{\infty}\|P_{k}x^{*}\|^{p^{*}}:x^{*}\in A\}=0.$$
\end{thm}
\begin{proof}
$(1)\Rightarrow (2)$. It is obvious that $P_{n}(A)$ is a $p^{*}$-($V$) set for each $n\in \mathbb{N}$. Let us assume that
$$\lim_{n\rightarrow \infty}\sup\{\sum_{k=n}^{\infty}\|P_{k}x^{*}\|^{p^{*}}:x^{*}\in A\}\neq 0.$$
By induction, we can find $\epsilon_{0}>0$, two sequences of positive integers $(p_{n})_{n}, (q_{n})_{n}$ with $p_{n}<q_{n}<p_{n+1}(n\in \mathbb{N})$ and a sequence $(x^{*}_{n})_{n}$ in $A$ such that $\sum_{k=p_{n}}^{q_{n}}\|P_{k}x^{*}_{n}\|^{p^{*}} > \epsilon_{0}$ for each $n\in \mathbb{N}$.
By Hahn-Banach Theorem, for each $n\in \mathbb{N}$, there exists a sequence $(x^{(n)}_{k})_{k=p_{n}}^{q_{n}}\in (\sum_{k=p_{n}}^{q_{n}}\oplus X_{k})_{p}$ such that $\sum_{k=p_{n}}^{q_{n}}\|x^{(n)}_{k}\|^{p}=1$ and
\begin{equation}\label{12}
\sum_{k=p_{n}}^{q_{n}}<P_{k}x^{*}_{n},x^{(n)}_{k}>=(\sum_{k=p_{n}}^{q_{n}}\|P_{k}x^{*}_{n}\|^{p^{*}})^{\frac{1}{p^{*}}}>\epsilon_{0}^{\frac{1}{p^{*}}}.
\end{equation}
For every $n\in \mathbb{N}$, we set $f_{n}\in X=(\sum_{k=1}^{\infty}\oplus X_{k})_{p}$ by

\[\pi_{k}(f_{n})= \left\{ \begin{array}
                    {r@{\quad,\quad}l}

 x^{(n)}_{k} & p_{n}\leq k\leq q_{n}\\ 0 & otherwise
 \end{array} \right. \]

Then the sequence $(f_{n})_{n}$ is weakly $p^{*}$-summable. Indeed, for every $x^{*}\in X^{*}$, we have
\begin{align*}
|<x^{*},f_{n}>|&=|\sum_{k=p_{n}}^{q_{n}}<P_{k}x^{*},x^{(n)}_{k}>|\\
&\leq (\sum_{k=p_{n}}^{q_{n}}\|x^{(n)}_{k}\|^{p})^{\frac{1}{p}}\cdot (\sum_{k=p_{n}}^{q_{n}}\|P_{k}x^{*}\|^{p^{*}})^{\frac{1}{p^{*}}}\\
&=(\sum_{k=p_{n}}^{q_{n}}\|P_{k}x^{*}\|^{p^{*}})^{\frac{1}{p^{*}}},
\end{align*}
which implies $$\sum_{n=1}^{\infty}|<x^{*},f_{n}>|^{p^{*}}\leq \sum_{n=1}^{\infty}\sum_{k=p_{n}}^{q_{n}}\|P_{k}x^{*}\|^{p^{*}}\leq
\sum_{k=1}^{\infty}\|P_{k}x^{*}\|^{p^{*}}<\infty.$$
By (1), we get $$\sum_{k=p_{n}}^{q_{n}}<P_{k}x^{*}_{n},x^{(n)}_{k}>=|<x^{*}_{n},f_{n}>|\leq \sup_{x^{*}\in A}|<x^{*},f_{n}>|\rightarrow 0 \quad(n\rightarrow \infty),$$
which contradicts with (\ref{12}).

$(2)\Rightarrow (1)$. Let $T$ be an operator from $l_{p}(c_{0}$ for $p^{*}=1$) into $X$. Then, by (2), we have
$$\sup_{x^{*}\in A}\|\sum_{k=1}^{n}T^{*}\circ J_{k}\circ P_{k}x^{*}-T^{*}x^{*}\|\leq \|T\|\cdot \sup_{x^{*}\in A}(\sum_{k=n+1}^{\infty}\|P_{k}x^{*}\|^{p^{*}})^{\frac{1}{p^{*}}}\rightarrow 0 \quad (n\rightarrow \infty).$$
Thus, for every $\epsilon>0$, there exists $n_{0}\in \mathbb{N}$ such that
$$T^{*}A \subset \sum_{k=1}^{n_{0}}T^{*}\circ J_{k}\circ P_{k}A+\epsilon B_{l_{p^{*}}}.$$
Since $P_{k}(A)$ is a $p^{*}$-($V$) set for each $k=1,2,...,n_{0}$, we get, by Theorem \ref{3.8}, that the subset $T^{*}\circ \pi^{*}_{k}\circ P_{k}A$ is relatively norm compact for each $k=1,2,...,n_{0}$ and so is $\sum_{k=1}^{n_{0}}T^{*}\circ J_{k}\circ P_{k}A$. Therefore, the subset $T^{*}A$ is relatively norm compact. Again by Theorem \ref{3.8}, we see that $A$ is a $p^{*}$-($V$) set.

\end{proof}

\begin{thm}\label{3.12}
Let $(X_{n})_{n}$ be a sequence of Banach spaces, $1<p<\infty$, $1\leq q<p^{*}$ and let $X=(\sum_{n=1}^{\infty}\oplus X_{n})_{p}$. Then a bounded subset $A$ of $X^{*}$ is a $q$-($V$) set if and only if each $P_{n}(A)$ does.
\end{thm}
\begin{proof}
We need only prove the sufficient part. Assume that $A$ is not a $q$-($V$) set. Then there exist $\epsilon_{0}>0$, a sequence $(x_{n})_{n}\in l^{w}_{q}(X)$ and a sequence $(x^{*}_{n})_{n}$ in $A$ such that
\begin{equation}\label{6}
|<x^{*}_{n},x_{n}>|=|\sum_{k=1}^{\infty}<P_{k}x^{*}_{n},\pi_{k}x_{n}>|>\epsilon_{0}, \quad n=1,2,...
\end{equation}

By the assumption, we get
\begin{equation}\label{5}
\lim_{n\rightarrow \infty}<P_{k}x^{*}_{n},\pi_{k}x_{n}>=0, \quad k=1,2,...
\end{equation}
By induction on $n$ in (\ref{6}) and $k$ in (\ref{5}), we get $$1=n_{1}<n_{2}<\cdots, \quad 0=k_{0}<k_{1}<k_{2}<\cdots,$$ such that
\begin{equation}\label{7}
|\sum_{k=k_{j}+1}^{\infty}<P_{k}x^{*}_{n_{j}},\pi_{k}x_{n_{j}}>|<\frac{\epsilon_{0}}{4}, \quad j=1,2,...
\end{equation}
and
\begin{equation}\label{20}
|\sum_{k=1}^{k_{j}}<P_{k}x^{*}_{n_{j+1}},\pi_{k}x_{n_{j+1}}>|<\frac{\epsilon_{0}}{4}, \quad j=1,2,...
\end{equation}
By (\ref{6}), (\ref{7}) and (\ref{20}), we get
$$|\sum_{k=k_{j-1}+1}^{k_{j}}<P_{k}x^{*}_{n_{j}},\pi_{k}x_{n_{j}}>|>\frac{\epsilon_{0}}{2}, \quad j=2,3,...$$
By (\ref{6}) and (\ref{7}), we get
$$|\sum_{k=1}^{k_{1}}<P_{k}x^{*}_{n_{1}},\pi_{k}x_{n_{1}}>|>\frac{3}{4}\epsilon_{0}>\frac{\epsilon_{0}}{2}.$$
Thus, we have
$$|\sum_{k=k_{j-1}+1}^{k_{j}}<P_{k}x^{*}_{n_{j}},\pi_{k}x_{n_{j}}>|>\frac{\epsilon_{0}}{2}, \quad j=1,2,...$$
For each $j=1,2,...,$ we set $y_{j}=x_{n_{j}}$ and $y^{*}_{j}\in X^{*}$ by

\[P_{k}y^{*}_{j}= \left\{ \begin{array}
                    {r@{\quad,\quad}l}

 P_{k}x^{*}_{n_{j}} & k_{j-1}+1\leq k\leq k_{j}\\ 0 & otherwise
 \end{array} \right. \]

Clearly, $(y_{j})_{j}\in l^{w}_{q}(X)$ and
$$|<y^{*}_{j},y_{j}>|=|\sum_{k=k_{j-1}+1}^{k_{j}}<P_{k}x^{*}_{n_{j}},\pi_{k}x_{n_{j}}>|>\frac{\epsilon_{0}}{2}, \quad j=1,2,...$$
Since the sequence $(y^{*}_{j})_{j}$ has pairwise disjoint supports, we see that $(y^{*}_{j})_{j}$ is equivalent to the unit vector basis $(e_{j})_{j}$ of $l_{p^{*}}$. Let $R$ be an isomorphic embedding from $l_{p^{*}}$ into $X^{*}$ with $Re_{j}=y^{*}_{j}(j=1,2,...)$. Let $T$ be an any operator from $l_{q^{*}}$ into $X$. By Pitt's Theorem, the operator $T^{*}R$ is compact and hence the sequence $(T^{*}y^{*}_{j})_{j}=(T^{*}Re_{j})_{j}$ is relatively norm compact. It follows from Theorem \ref{3.8} that the sequence $(y^{*}_{j})_{j}$ is
a $q$-($V$) set. Since $(y_{j})_{j}$ is weakly $q$-summable, we have
$$|<y^{*}_{n},y_{n}>|\leq \sup_{j}|<y^{*}_{j},y_{n}>|\rightarrow 0 \quad (n\rightarrow \infty),$$
this contradiction concludes the proof.

\end{proof}

The following two lemmas are well-known (see \cite{B}, for example).
\begin{lem}\label{3.14}
Let $(X_{n})_{n}$ be a sequence of Banach spaces. The following are equivalent about a bounded subset $A$ of $(\sum_{n=1}^{\infty}\oplus X_{n})_{1}$:
\item[(1)]$A$ is relatively weakly compact;
\item[(2)]$\pi_{n}(A)$ is relatively weakly compact for each $n\in \mathbb{N}$ and $$\lim_{n\rightarrow \infty}\sup\{\sum_{k=n}^{\infty}\|\pi_{k}x\|:x\in A\}=0.$$
\end{lem}

\begin{lem}\label{3.15}
Let $(X_{n})_{n}$ be a sequence of Banach spaces and let $X=(\sum_{n=1}^{\infty}\oplus X_{n})_{p}$ ($1<p<\infty$) or $X=(\sum_{n=1}^{\infty}\oplus X_{n})_{0}$. Then a bounded subset $A$ of $X$ is relatively weakly compact if and only if every $\pi_{n}(A)$ does.
\end{lem}

\begin{thm}
Let $(X_{n})_{n}$ be a sequence of Banach spaces, $1\leq q<\infty$ and $1<p<\infty$. Then $(\sum_{n=1}^{\infty}\oplus X_{n})_{p}$ has property $q$-($V$) if and only if each $X_{n}$ does.
\end{thm}
\begin{proof}
The necessary part follows from Corollary \ref{3.6}.

Conversely, let $A$ be a $q$-($V$) subset of $(\sum_{n=1}^{\infty}\oplus X^{*}_{n})_{p^{*}}$. Then each $P_{n}A$ is
also a $q$-($V$) set. By hypothesis, each $P_{n}A$ is relatively weakly compact. It follows from Lemma \ref{3.15} that $A$ is relatively weakly compact. This concludes the proof.

\end{proof}

\begin{thm}
Let $(X_{n})_{n}$ be a sequence of Banach spaces. Then $(\sum_{n=1}^{\infty}\oplus X_{n})_{0}$ has property $1$-($V$) if and only if each $X_{n}$ does.
\end{thm}
\begin{proof}
The necessary part follows from Corollary \ref{3.6}.

Conversely, assume that $A$ is a $1$-($V$) subset of $(\sum_{n=1}^{\infty}\oplus X^{*}_{n})_{1}$. By Theorem \ref{3.11},
each $P_{n}(A)$ is a $1$-($V$) set and $$\lim_{n\rightarrow \infty}\sup\{\sum_{k=n}^{\infty}\|P_{k}x^{*}\|:x^{*}\in A\}=0.$$
By the assumption, each $P_{n}A$ is relatively weakly compact. It follows from Lemma \ref{3.14} that $A$ is relatively weakly compact. We are done.

\end{proof}

Combining Theorem \ref{3.4} with the same argument as Theorem \ref{3.11}, we obtain the similar result for the $p$-($V^{*}$) sets.

\begin{thm}\label{3.100}
Let $(X_{n})_{n}$ be a sequence of Banach spaces. Let $A$ be a bounded subset of $X=(\sum_{n=1}^{\infty}\oplus X_{n})_{p}(1\leq p<\infty)$. The following assertions are equivalent:
\item[(1)]$A$ is a $p$-($V^{*}$) set;
\item[(2)]$\pi_{n}(A)$ is a $p$-($V^{*}$) set for each $n\in \mathbb{N}$ and $$\lim_{n\rightarrow \infty}\sup\{\sum_{k=n}^{\infty}\|\pi_{k}x\|^{p}:x\in A\}=0.$$
\end{thm}

\begin{thm}
Let $(X_{n})_{n}$ be a sequence of Banach spaces and $1\leq q<p<\infty$. Let $A$ be a bounded subset of $X=(\sum_{n=1}^{\infty}\oplus X_{n})_{p}$ or $X=(\sum_{n=1}^{\infty}\oplus X_{n})_{0}$. Then $A$ is a $q$-($V^{*}$) set if and only if each $\pi_{n}(A)$ does.
\end{thm}
The proof is similar to Theorem \ref{3.12}, only interchanging the role of $X$ and $X^{*}$ and replacing Theorem \ref{3.8} by Theorem \ref{3.4}.

\begin{thm}\label{3.13}
Let $(X_{n})_{n}$ be a sequence of Banach spaces, $1\leq q<\infty$, $1<p<\infty$ and let $X=(\sum_{n=1}^{\infty}\oplus X_{n})_{p}$ or $X=(\sum_{n=1}^{\infty}\oplus X_{n})_{0}$. Then $X$ has property $q$-($V^{*}$) if and only if each $X_{n}$ does.
\end{thm}
\begin{proof}
The necessary part follows from Corollary \ref{3.10}.

Conversely, let $A$ be a $q$-($V^{*}$) subset of $X$. Clearly, each $\pi_{n}(A)$ is a $q$-($V^{*}$) subset of $X_{n}$. By the assumption, each $\pi_{n}(A)$ is relatively weakly compact. It follows from Lemma \ref{3.15} that $A$ is relatively weakly compact. Therefore, $X$ has property $q$-($V^{*}$).

\end{proof}

\begin{thm}\label{3.201}
Let $(X_{n})_{n}$ be a sequence of Banach spaces and $1\leq p<\infty$. Then $(\sum_{n=1}^{\infty}\oplus X_{n})_{p}$ has property $p$-($V^{*}$) if and only if so does each $X_{n}$.
\end{thm}
\begin{proof}
The necessary part follows from Corollary \ref{3.10}.

Conversely, let $A$ be a $p$-($V^{*}$) set. It follows from Theorem \ref{3.100} that each $\pi_{n}(A)$ is a $p$-($V^{*}$) set and $$\lim_{n\rightarrow \infty}\sup\{\sum_{k=n}^{\infty}\|\pi_{k}x\|^{p}:x\in A\}=0.$$
Since each $X_{n}$ has property $p$-($V^{*}$), each $\pi_{n}(A)$ is relatively weakly compact.
For $1<p<\infty$, it follows from Lemma \ref{3.15} that $A$ is relatively weakly compact. For $p=1$, Lemma \ref{3.14} yields that $A$ is relatively weakly compact.
Thus, in both cases, $A$ is relatively weakly compact. This concludes the proof.

\end{proof}

\section{Quantifying Pe{\l}czy\'{n}ski's property ($V$) of order $p$ and Pe{\l}czy\'{n}ski's property ($V^{*}$) of order $p$}
We will need several measures of weak non-compactness.
Let $A$ be a bounded subset of a Banach space $X$.
Other commonly used quantities measuring weak non-compactness are:
\begin{center}
$wk_{X}(A)=\widehat{d}(\overline{A}^{w^{*}},X),$ where $\overline{A}^{w^{*}}$ denotes the $weak^{*}$ closure of $A$ in $X^{**}$.
\end{center}

\begin{center}
$wck_{X}(A)=\sup\{d(clust_{X^{**}}((x_{n})_{n}),X):(x_{n})_{n}$ is a sequence in $A\}$, where $clust_{X^{**}}((x_{n})_{n})$ is the set of all $weak^{*}$ cluster points in $X^{**}$ of $(x_{n})_{n}$.
\end{center}

\begin{center}
$\gamma_{X}(A)=\sup\{|\lim_{n}\lim_{m}<x^{*}_{m},x_{n}>-\lim_{m}\lim_{n}<x^{*}_{m},x_{n}>|:(x_{n})_{n}$ is a sequence in $A$, $(x^{*}_{m})_{m}$ is a sequence in $B_{X^{*}}$ and all the involved limits exist$\}$.
\end{center}
It follows from \cite[Theorem 2.3]{AC} that for any bounded subset $A$ of a Banach space $X$ we have
\begin{equation}\label{8}
wck_{X}(A)\leq wk_{X}(A)\leq \gamma_{X}(A)\leq 2wck_{X}(A),
\end{equation}
$$wk_{X}(A)\leq \omega(A).$$
For an operator $T: X\rightarrow Y$, $\omega(T),wk_{Y}(T),wck_{Y}(T),\gamma_{Y}(T)$ will denote $\omega(TB_{X}),wk_{Y}(TB_{X})$,
$wck_{Y}(TB_{X})$ and $\gamma_{Y}(TB_{X})$, respectively.
C. Angosto and B. Cascales(\cite{AC})proved the following inequality:
\begin{center}
$\gamma_{Y}(T)\leq \gamma_{X^{*}}(T^{*})\leq 2\gamma_{Y}(T)$, for any operator $T$.
\end{center}
So,putting these inequalities together, we get,for any operator $T$,
\begin{equation}\label{1}
{\frac{1}{2}}wk_{Y}(T)\leq wk_{X^{*}}(T^{*})\leq 4wk_{Y}(T).
\end{equation}
For an operator $T: X\rightarrow Y$ and $1\leq p\leq\infty$. We set
\begin{center}
$uc_{p}(T)=\sup\{\limsup_{n}\|Tx_{n}\|: (x_{n})_{n}\in l^{w}_{p}(X), \|(x_{n})_{n}\|_{p}^{w}\leq 1\},$
\end{center}

We begin this section with a simple lemma in \cite{CCL}, which will be used frequently.
\begin{lem}\cite{CCL}\label{5.4}
Let $X$ be a closed subspace of a Banach space $Y$ and let $A$ be a bounded subset of $X$. Then
\begin{equation}\label{2}
wk_{Y}(A)\leq wk_{X}(A)\leq 2wk_{Y}(A).
\end{equation}
\end{lem}
It is worth mentioning that the constant 2 in the right inequality of (\ref{2}) is optimal. Indeed, let $X=c_{0}, Y=l_{\infty}$ and $A$ be the summing basis of $c_{0}$. It is easy to check that $wk_{X}(A)=1$ and $wk_{Y}(A)=\frac{1}{2}$.

\begin{defn}
Let $1\leq p\leq\infty$. We say that a Banach space $X$ has \textit{quantitative Pe{\l}czy\'{n}ski's property $(V^{*})$ of order $p$} (property $p$-$(V^{*})_{q}$ in short) with a constant $C>0$ if for every Banach space $Y$ and every operator $T: Y\rightarrow X$, one has $$wk_{X}(T)\leq C\cdot uc_{p}(T^{*}).$$
We say that a Banach space $X$ has property $p$-$(V^{*})_{q}$ if it has property $p$-$(V^{*})_{q}$ with some constant $C$.
\end{defn}

Let $1\leq p\leq\infty$ and $X$ be a Banach space. For a bounded subset $A$ of $X$, we set
$$\iota_{p}(A)=\sup\{\limsup_{n}\sup_{x\in A}|<x^{*}_{n},x>|: (x^{*}_{n})_{n}\in l^{w}_{p}(X^{*}), \|(x^{*}_{n})_{n}\|_{p}^{w}\leq 1\}.$$

\begin{thm}\label{5.10}
Let $X$ be a Banach space and $1\leq p<\infty$. The following statements are equivalent:
\item[(1)]$X$ has property $p$-$(V^{*})_{q}$;
\item[(2)]there exists a constant $C>0$ such that for each bounded subset $A$ of $X$, one has
$$wk_{X}(A)\leq C\cdot\iota_{p}(A).$$
\end{thm}
\begin{proof}
$(1)\Rightarrow (2)$. Suppose that $X$ has property $p$-$(V^{*})_{q}$ with a constant $C>0$.
Let $A$ be a bounded subset of $X$.
We first claim that $$wck_{X}(A)\leq C\cdot\iota_{p}(A).$$
Indeed, we may assume that $wck_{X}(A)>0$ and fix an arbitrary $\epsilon\in (0, wck_{X}(A))$. Then there exists a sequence $(x_{n})_{n}$ in $A$ such that $\epsilon<wck_{X}((x_{n})_{n})$. Define an operator
$$T: l_{1}\rightarrow X, \quad (\alpha_{n})_{n}\mapsto \sum_{n=1}^{\infty}\alpha_{n}x_{n}, \quad (\alpha_{n})_{n}\in l_{1}.$$
By (1) and (\ref{8}), we get
$$wck_{X}((x_{n})_{n})\leq wck_{X}(T)\leq wk_{X}(T)\leq C\cdot uc_{p}(T^{*}).$$
By the definitions of $T$ and $\iota_{p}(A)$, we get $$uc_{p}(T^{*})\leq \iota_{p}(A).$$ This yields $\epsilon<C\cdot\iota_{p}(A).$ By the arbitrariness of $\epsilon\in (0, wck_{X}(A))$, we prove the claim. Again by (\ref{8}), we obtain $$wk_{X}(A)\leq 2C\cdot\iota_{p}(A).$$

$(2)\Rightarrow (1)$ is trivial by taking $A=TB_{Y}$ with the same constant $C$.

\end{proof}

\begin{thm}\label{5.5}
Let $1\leq p\leq \infty$. If a Banach space $X$ has property $p$-$(V^{*})_{q}$ with a constant $C$, then every closed subspace of $X$ has property $p$-$(V^{*})_{q}$ with $2C$.
\end{thm}
\begin{proof}
Let $M$ be a closed subspace of $X$. Suppose that $X$ has property $p$-$(V^{*})_{q}$ with a constant $C>0$. We'll show that $M$ has property $p$-$(V^{*})_{q}$ with $2C$. Fix a Banach space $Y$ and an operator $S: Y \rightarrow M$. Then we have
$$wk_{X}(iS)\leq C\cdot uc_{p}(S^{*}i^{*}),$$ where $i: M\rightarrow X$ is the inclusion map.
By (\ref{2}), we get $$wk_{M}(S)\leq 2wk_{X}(iS)\leq 2C\cdot uc_{p}(S^{*}i^{*})\leq 2C\cdot uc_{p}(S^{*}),$$
we are done.

\end{proof}

\begin{thm}
Let $1\leq p\leq \infty$. Then a Banach space $X$ has property $p$-$(V^{*})_{q}$ if and only if there exists a constant $C>0$ such that every separable closed subspace of $X$ has property $p$-$(V^{*})_{q}$ with $C$.
\end{thm}
\begin{proof}
The necessary part follows from Theorem \ref{5.5}. Conversely, let $C>0$ be such that every separable closed subspace of $X$ has property $p$-$(V^{*})_{q}$ with $C$. We claim that
\begin{center}
$\gamma_{X}(T)\leq 2C\cdot uc_{p}(T^{*})$, for every Banach space $Y$ and every operator $T: Y\rightarrow X$.
\end{center}
Fix a space $Y$ and an operator $T: Y\rightarrow X$.
Let $A=TB_{Y}$. We may assume that $\gamma_{X}(A)>0$. Fix any $\epsilon\in(0,\gamma_{X}(A))$. Then there exists a sequence $(x_{n})_{n}$ in $A$ such that $\epsilon<\gamma_{X}((x_{n})_{n})$. By \cite[Proposition 3.4]{HM}, there exist a separable closed subspace $Z$ of $X$ that contains $(x_{n})_{n}$ and an isometric embedding $J: Z^{*}\rightarrow X^{*}$ such that $Jz^{*}|_{Z}=z^{*}$ for every $z^{*}\in Z^{*}$. Define an operator
$$P: X^{*}\rightarrow J(Z^{*}),  \quad x^{*}\mapsto J(x^{*}|_{Z}), \quad x^{*}\in X^{*}.$$
Then $P$ is a linear projection from $X^{*}$ onto $J(Z^{*})$ with $\|P\|=1$. We define an operator
$$S: l_{1}\rightarrow Z, \quad (\alpha_{n})_{n}\mapsto \sum_{n=1}^{\infty}\alpha_{n}x_{n}, \quad (\alpha_{n})_{n}\in l_{1}.$$
By hypothesis, we get $$wk_{Z}(S)\leq C\cdot uc_{p}(S^{*}).$$ By (\ref{8}), we have
$$wk_{Z}(S)\geq wk_{Z}((x_{n})_{n})\geq \frac{1}{2}\gamma_{Z}((x_{n})_{n}).$$
By the definition of $S$ and the properties of $J$, we obtain
\begin{align*}
\gamma_{Z}((x_{n})_{n})&\leq 2C\cdot \sup\{\limsup_{n}\sup_{k}|<z^{*}_{n},x_{k}>|: (z^{*}_{n})_{n}\in l^{w}_{p}(Z^{*}), \|(z^{*}_{n})_{n}\|_{p}^{w}\leq 1\}\\
&\leq 2C\cdot \sup\{\limsup_{n}\sup_{k}|<x^{*}_{n},x_{k}>|: (x^{*}_{n})_{n}\in l^{w}_{p}(X^{*}), \|(x^{*}_{n})_{n}\|_{p}^{w}\leq 1\}\\
&\leq 2C\cdot uc_{p}(T^{*})\\
\end{align*}
By the definition of $P$, we get $$\gamma_{X}((x_{n})_{n})=\gamma_{Z}((x_{n})_{n}).$$ Thus, one has
$\epsilon<2C\cdot uc_{p}(T^{*}),$ which proves the claim by the arbitrariness of $\epsilon$.

By (\ref{8}) again, we have $$wk_{X}(T)\leq 2C\cdot uc_{p}(T^{*}).$$ This implies that
$X$ has property $p$-$(V^{*})_{q}$ with $2C$.

\end{proof}

\begin{thm}\label{5.8}
Let $X=L_{1}(\mu,\mathbb{R})$, where $(\Omega,\Sigma,\mu)$ is a finite measure space. Then $$wk_{X}(A)=\iota_{1}(A)$$ for each bounded subset $A$ of $X$.
\end{thm}
\begin{proof}
We may assume that $A$ is a subset of $B_{X}$.

Step 1. $wk_{X}(A)\leq \iota_{1}(A)$.

Let us assume that $wk_{X}(A)>0$ and fix an arbitrary $\epsilon\in (0,wk_{X}(A))$. It follows from \cite[Proposition 7.1]{KKS} that there exists a sequence $(x_{k})_{k}$ in $A$ such that
\begin{equation}\label{30}
\int_{E_{k}}|x_{k}|d\mu>\epsilon, \quad k=1,2,\cdots
\end{equation}
where $E_{k}=\{t\in \Omega: |x_{k}(t)|\geq k\}$.

For each $k\in \mathbb{N}$, Chebyshev's inequality gives $\mu(E_{k})\leq \frac{1}{k}$ and hence the sequence $(x_{k}\chi_{E_{k}})_{k}$ converges to $0$ in measure. By \cite[Lemma 5.2.1]{AK}, there exist a subsequence $(x_{k_{n}}\chi_{E_{k_{n}}})_{n}$ of $(x_{k}\chi_{E_{k}})_{k}$ and a sequence of disjoint measurable sets $(A_{n})_{n}$ such that
\begin{equation}\label{31}
\|x_{k_{n}}\chi_{E_{k_{n}}}-x_{k_{n}}\chi_{E_{k_{n}}}\chi_{A_{n}}\|_{1}\rightarrow 0 \quad (n\rightarrow \infty).
\end{equation}
Set $B_{n}=E_{k_{n}}\cap A_{n}$ and $f_{n}=sign(x_{k_{n}})\chi_{B_{n}}$ for each $n\in \mathbb{N}$. Then $(f_{n})_{n}$ is weakly 1-summable in $X^{*}$ and $\|(f_{n})_{n}\|_{1}^{w}\leq 1$. Combining (\ref{30}) with (\ref{31}), we get $$\limsup_{n}|<f_{n},x_{k_{n}}>|\geq \epsilon,$$
which implies that $\iota_{1}(A)\geq \epsilon$. Since $\epsilon\in (0,wk_{X}(A))$ is arbitrary, we conclude Step 1.

Step 2. $\iota_{1}(A)\leq wk_{X}(A)$.

Similarly, we can assume that $\iota_{1}(A)>0$ and fix an arbitrary $\epsilon\in (0,\iota_{1}(A))$. Then there exist a sequence $(x_{n})_{n}$ in $A$ and $(f_{n})_{n}\in l^{w}_{1}(X^{*})$ with $\|(f_{n})_{n}\|_{1}^{w}\leq 1$ such that
\begin{equation}\label{32}
\int_{\Omega}f_{n}\cdot x_{n}d\mu>\epsilon, \quad n=1,2,\cdots
\end{equation}
It is easy to verify that $$\|(f_{n})_{n}\|_{1}^{w}=\sup_{n}\|\sum_{k=1}^{n}|f_{k}|\|.$$ Thus, we get $$\sum_{k=1}^{\infty}|f_{k}|\leq 1, \mu-a.e.$$
For the sake of convenience, we may assume that $\sum_{k=1}^{\infty}|f_{k}|\leq 1$ everywhere.

Let $\delta>0$ be arbitrary. By (\ref{32}), we obtain $N_{1}\in \mathbb{N}$ such that $$\int_{E_{1}}f_{1}\cdot x_{1}d\mu>\epsilon,$$
where $E_{1}=\{t\in \Omega: \sum_{k=N_{1}}^{\infty}|f_{k}(t)|<\frac{\delta}{2}\}$.

Set $\widetilde{f_{1}}=f_{1}\cdot \chi_{E_{1}}$ and $\widetilde{f}_{n}=f_{n}\cdot \chi_{E^{c}_{1}}(n\geq N_{1})$. Then $\widetilde{f_{1}}$ is disjoint from $\widetilde{f}_{n}$ for each $n\geq N_{1}$. Moreover, for each $n\geq N_{1}$, we have
\begin{equation}\label{33}
\int_{\Omega}\widetilde{f}_{n}\cdot x_{n}d\mu=\int_{\Omega}f_{n}\cdot x_{n}d\mu-\int_{E_{1}}f_{n}\cdot x_{n}d\mu>\epsilon-\frac{\delta}{2}.
\end{equation}
Obviously,
\begin{center}
$\sum_{k=N_{1}}^{\infty}|\widetilde{f_{k}}|\leq \sum_{k=N_{1}}^{\infty}|f_{k}|\leq 1$ everywhere.
\end{center}
By (\ref{33}), we obtain $N_{2}>N_{1}$ such that
\begin{equation}\label{34}
\int_{E_{2}}\widetilde{f}_{N_{1}}\cdot x_{N_{1}}d\mu >\epsilon-\frac{\delta}{2}.
\end{equation}
where $E_{2}=\{t\in \Omega: \sum_{k=N_{2}}^{\infty}|\widetilde{f_{k}}(t)|<\frac{\delta}{2^{2}}\}$.

Set $\widetilde{\widetilde{f}_{2}}=\widetilde{f}_{N_{1}}\cdot \chi_{E_{2}}$ and ${\widetilde{\widetilde{f}}}_{n}=\widetilde{f}_{n}\cdot \chi_{E^{c}_{2}}(n\geq N_{2})$. Then, by (\ref{34}), we get
\begin{equation}\label{35}
\int_{F_{1}}|x_{N_{1}}|d\mu\geq \int_{\Omega}\widetilde{\widetilde{f}_{2}}\cdot x_{N_{1}}d\mu=\int_{E_{2}}\widetilde{f}_{N_{1}}\cdot x_{N_{1}}d\mu >\epsilon-\frac{\delta}{2}.
\end{equation}
where $F_{1}$ is the support of $\widetilde{\widetilde{f}_{2}}$.

By (\ref{33}), for each $n\geq N_{2}$, we have
\begin{align*}
\int_{\Omega}{\widetilde{\widetilde{f}}}_{n}\cdot x_{n}d\mu&=\int_{E^{c}_{2}}\widetilde{f}_{n}\cdot x_{n}d\mu\\
&=\int_{\Omega}\widetilde{f}_{n}\cdot x_{n}d\mu-\int_{E_{2}}\widetilde{f}_{n}\cdot x_{n}d\mu\\
&>\epsilon-\frac{\delta}{2}-\frac{\delta}{2^{2}}.
\end{align*}
In a similar way, we obtain $N_{3}>N_{2}$ such that
\begin{equation}\label{36}
\int_{E_{3}}\widetilde{\widetilde{f}}_{N_{2}}\cdot x_{N_{2}}d\mu >\epsilon-\frac{\delta}{2}-\frac{\delta}{2^{2}}.
\end{equation}
where $E_{3}=\{t\in \Omega: \sum_{k=N_{3}}^{\infty}|\widetilde{\widetilde{f}_{k}}(t)|<\frac{\delta}{2^{3}}\}$.

Set $\widetilde{\widetilde{\widetilde{f}_{3}}}=\widetilde{\widetilde{f}}_{N_{2}}\cdot \chi_{E_{3}}$ and ${\widetilde{\widetilde{\widetilde{f}}}}_{n}={\widetilde{\widetilde{f}}}_{n}\cdot \chi_{E^{c}_{3}}(n\geq N_{3})$.
Then, by (\ref{36}), we get
\begin{equation}\label{37}
\int_{F_{2}}|x_{N_{2}}|d\mu\geq \int_{\Omega}\widetilde{\widetilde{\widetilde{f}_{3}}}\cdot x_{N_{2}}d\mu >\epsilon-\frac{\delta}{2}-\frac{\delta}{2^{2}}.
\end{equation}
where $F_{2}$ is the support of $\widetilde{\widetilde{\widetilde{f}_{3}}}$.

Since $\widetilde{\widetilde{f}_{2}}$ is disjoint from ${\widetilde{\widetilde{f}}}_{n}$ for each $n\geq N_{2}$, the set $F_{2}$ is also disjoint from $F_{1}$.

A similar computation shows that for each $n\geq N_{3}$,
$$\int_{\Omega}{\widetilde{\widetilde{\widetilde{f}}}}_{n}\cdot x_{n}d\mu>\epsilon-\frac{\delta}{2}-\frac{\delta}{2^{2}}-\frac{\delta}{2^{3}}.$$
Thus, there exists $N_{4}>N_{3}$ such that
\begin{equation}\label{38}
\int_{E_{4}}\widetilde{\widetilde{\widetilde{f}}}_{N_{3}}\cdot x_{N_{3}}d\mu >\epsilon-\frac{\delta}{2}-\frac{\delta}{2^{2}}-\frac{\delta}{2^{3}}.
\end{equation}
where $E_{4}=\{t\in \Omega: \sum_{k=N_{4}}^{\infty}|\widetilde{\widetilde{\widetilde{f}_{k}}}(t)|<\frac{\delta}{2^{4}}\}$.

Set $\widetilde{\widetilde{\widetilde{\widetilde{f}_{4}}}}=\widetilde{\widetilde{\widetilde{f}}}_{N_{3}}\cdot \chi_{E_{4}}$ and
${\widetilde{\widetilde{\widetilde{\widetilde{f}}}}}_{n}={\widetilde{\widetilde{\widetilde{f}}}}_{n}\cdot \chi_{E_{4}^{c}}$ for $n\geq N_{4}$.
Inequality (\ref{38}) yields
$$\int_{F_{3}}|x_{N_{3}}|d\mu\geq \int_{\Omega}\widetilde{\widetilde{\widetilde{\widetilde{f}_{4}}}}\cdot x_{N_{3}}d\mu>
\epsilon-\frac{\delta}{2}-\frac{\delta}{2^{2}}-\frac{\delta}{2^{3}},$$
where $F_{3}$ is the support of $\widetilde{\widetilde{\widetilde{\widetilde{f}_{4}}}}$.

Since $\widetilde{\widetilde{\widetilde{f}_{3}}}$ is disjoint from $\widetilde{\widetilde{\widetilde{f}}}_{N_{3}}$, the set $F_{3}$ is disjoint from $F_{2}$. Since $\widetilde{\widetilde{f}_{2}}$ is disjoint from $\widetilde{\widetilde{f}}_{N_{3}}$, the set $F_{3}$ is disjoint from $F_{1}$.

By induction, we get a subsequence $(x_{N_{k}})_{k}$ of $(x_{n})_{n}$ and a sequence of disjoint measurable sets $(F_{k})_{k}$ such that
$$\int_{F_{k}}|x_{N_{k}}|d\mu>\epsilon-\sum_{i=1}^{k}\frac{\delta}{2^{i}}>\epsilon-\delta, \quad k=1,2,\cdots$$

Claim: $wk_{X}(A)\geq \epsilon-\delta.$

Indeed, fix any $c>0$. Then, for each $k\in \mathbb{N}$,
\begin{align*}
\epsilon-\delta&<\int_{F_{k}}|x_{N_{k}}|d\mu\\
&=\int_{F_{k}\cap \{|x_{N_{k}}|\geq c\}}(|x_{N_{k}}|-c)d\mu+c\mu(F_{k}\cap \{|x_{N_{k}}|\geq c\})+\int_{F_{k}\cap \{|x_{N_{k}}|<c\}}|x_{N_{k}}|d\mu\\
&\leq \int_{F_{k}\cap \{|x_{N_{k}}|\geq c\}}(|x_{N_{k}}|-c)d\mu+c\mu(F_{k}\cap \{|x_{N_{k}}|\geq c\})+c\mu(F_{k}\cap \{|x_{N_{k}}|<c\})\\
&\leq \int_{\{|x_{N_{k}}|\geq c\}}(|x_{N_{k}}|-c)d\mu+c\mu(F_{k})\\
&\leq \sup_{x\in A}\int_{\{|x|\geq c\}}(|x|-c)d\mu+c\mu(F_{k})\\
\end{align*}
Since $(F_{k})_{k}$ is disjoint and $\mu$ is finite, $\mu(F_{k})\rightarrow 0(k\rightarrow \infty)$.

Letting $k\rightarrow \infty$, we get
$$\epsilon-\delta\leq \sup_{x\in A}\int_{\{|x|\geq c\}}(|x|-c)d\mu.$$
Again by \cite[Proposition 7.1]{KKS}, we prove the claim.

Since $\delta>0$ is arbitrary, we get $$wk_{X}(A)\geq \epsilon.$$
By the arbitrariness of $\epsilon\in (0,\iota_{1}(A))$, we obtain $$\iota_{1}(A)\leq wk_{X}(A).$$
This completes the proof.

\end{proof}

\begin{thm}\label{5.12}
Let $X=l_{1}(\mathbb{N}, \mathbb{R})$. Then $$wk_{X}(A)=\iota_{1}(A)$$ for each bounded subset $A$ of $X$.
\end{thm}
\begin{proof}
We may assume that $A$ is a subset of $B_{X}$.

Step 1. $wk_{X}(A)\leq \iota_{1}(A)$.

We may assume that $wk_{X}(A)>0$ and fix an arbitrary $c\in (0,wk_{X}(A))$. By \cite[Lemma 7.2]{KKS}, we obtain two sequences $(p_{n})_{n}, (q_{n})_{n}$ of natural numbers with $p_{n}<q_{n}<p_{n+1}(n\in \mathbb{N})$ and a sequence $(x_{n})_{n}$ in $A$ such that
$$\sum_{k=p_{n}}^{q_{n}}|x_{n}(k)|>c, \quad n=1,2\cdots$$
By Hahn-Banach Theorem, for each $n\in \mathbb{N}$, we can find $(\alpha_{n}(k))_{k=p_{n}}^{q_{n}}$ with
$\sup_{p_{n}\leq k\leq q_{n}}|\alpha_{n}(k)|=1$ such that
\begin{equation}\label{22}
\sum_{k=p_{n}}^{q_{n}}x_{n}(k)\alpha_{n}(k)>c.
\end{equation}
For each $n\in \mathbb{N}$, we set $f_{n}\in l_{\infty}$ by

\[f_{n}(k)= \left\{ \begin{array}
                    {r@{\quad,\quad}l}

 \alpha_{n}(k) & p_{n}\leq k\leq q_{n}\\ 0 & otherwise
 \end{array} \right. \]

Then $(f_{n})_{n}$ is weakly $1$-summable and $\|(f_{n})_{n}\|_{1}^{w}\leq 1$. Indeed, for each $x\in X$, we have

\begin{align*}
\sum_{n=1}^{\infty}|<f_{n},x>|&=\sum_{n=1}^{\infty}|\sum_{k=p_{n}}^{q_{n}}x(k)\alpha_{n}(k)|\\
&\leq \sum_{n=1}^{\infty}\sum_{k=p_{n}}^{q_{n}}|x(k)||\alpha_{n}(k)|\\
&\leq \sum_{n=1}^{\infty}\sum_{k=p_{n}}^{q_{n}}|x(k)|\leq \|x\|.\\
\end{align*}
It follows from (\ref{22}) that
$$\sup_{x\in A}|<f_{n},x>|\geq |<f_{n},x_{n}>|=\sum_{k=p_{n}}^{q_{n}}x_{n}(k)\alpha_{n}(k)>c, \quad n=1,2,\ldots.$$
This implies that $\iota_{1}(A)\geq c$.

Since $c\in (0,wk_{X}(A))$ is arbitrary, we get $$wk_{X}(A)\leq \iota_{1}(A).$$

Step 2. $\iota_{1}(A)\leq wk_{X}(A)$.

Assume that $\iota_{1}(A)>0$ and fix an arbitrary $\epsilon\in (0,\iota_{1}(A))$. Then there exist a sequence $(x_{n})_{n}$ in $A$ and $(f_{n})_{n}\in l^{w}_{1}(X^{*})$ with $\|(f_{n})_{n}\|_{1}^{w}\leq 1$ such that
\begin{equation}\label{50}
\sum_{k=1}^{\infty}f_{n}(k)x_{n}(k)>\epsilon, \quad n=1,2,\cdots
\end{equation}
It follows from $\|(f_{n})_{n}\|_{1}^{w}\leq 1$ that
\begin{equation}\label{52}
\sum_{n=1}^{\infty}|f_{n}(k)|\leq 1, \quad k=1,2,\cdots
\end{equation}
Let $\delta>0$ be arbitrary. By (\ref{50}) and (\ref{52}), there exists $N_{1}\in \mathbb{N}$ such that
$$\sum_{k\in E_{1}}f_{1}(k)x_{1}(k)>\epsilon,$$
where $E_{1}=\{k: \sum_{n=N_{1}}^{\infty}|f_{n}(k)|<\frac{\delta}{2}\}$.

Set $\widetilde{f_{1}}=f_{1}\cdot \chi_{E_{1}}$ and $\widetilde{f}_{n}=f_{n}\cdot \chi_{E^{c}_{1}}(n\geq N_{1})$. By (\ref{50}) and the definition of $E_{1}$, we have
\begin{equation}\label{51}
<\widetilde{f}_{n},x_{n}>=\sum_{k=1}^{\infty}f_{n}(k)x_{n}(k)-\sum_{k\in E_{1}}f_{n}(k)x_{n}(k)>\epsilon-\frac{\delta}{2},
\end{equation}
for each $n\geq N_{1}$.

Inequality (\ref{52}) yields that
\begin{equation}\label{53}
\sum_{n=N_{1}}^{\infty}|\widetilde{f}_{n}(k)|\leq\sum_{n=N_{1}}^{\infty}|f_{n}(k)|\leq 1, \quad k=1,2,\cdots
\end{equation}
Combining (\ref{51}) with (\ref{53}), we obtain $N_{2}>N_{1}$ such that
\begin{equation}\label{54}
\sum_{k\in E_{2}}\widetilde{f}_{N_{1}}(k)x_{N_{1}}(k)>\epsilon-\frac{\delta}{2},
\end{equation}
where $E_{2}=\{k:\sum_{n=N_{2}}^{\infty}|\widetilde{f}_{n}(k)|<\frac{\delta}{2^{2}}\}$.

Set $\widetilde{\widetilde{f}_{2}}=\widetilde{f}_{N_{1}}\cdot \chi_{E_{2}}$ and ${\widetilde{\widetilde{f}}}_{n}=\widetilde{f}_{n}\cdot \chi_{E^{c}_{2}}(n\geq N_{2})$. Then, by (\ref{54}), we get
$$\sum_{k\in F_{1}}|x_{N_{1}}(k)|\geq \sum_{k=1}^{\infty}\widetilde{\widetilde{f}_{2}}(k)x_{N_{1}}(k)>\epsilon-\frac{\delta}{2},$$
where $F_{1}$ is the support of $\widetilde{\widetilde{f}_{2}}$.

Moreover, by (\ref{51}) and the definition of $E_{2}$, we get
\begin{equation}\label{55}
<{\widetilde{\widetilde{f}}}_{n},x_{n}>=\sum_{k=1}^{\infty}\widetilde{f}_{n}(k)x_{n}(k)-
\sum_{k\in E_{2}}\widetilde{f}_{n}(k)x_{n}(k)>\epsilon-\frac{\delta}{2}-\frac{\delta}{2^{2}},
\end{equation}
for each $n\geq N_{2}$.

By (\ref{55}), there exists $N_{3}>N_{2}$ such that
\begin{equation}\label{56}
\sum_{k\in E_{3}}{\widetilde{\widetilde{f}}}_{N_{2}}(k)x_{N_{2}}(k)>\epsilon-\frac{\delta}{2}-\frac{\delta}{2^{2}},
\end{equation}
where $E_{3}=\{k: \sum_{n=N_{3}}^{\infty}|{\widetilde{\widetilde{f}}}_{n}(k)|<\frac{\delta}{2^{3}}\}$.

Set $\widetilde{\widetilde{\widetilde{f}_{3}}}=\widetilde{\widetilde{f}}_{N_{2}}\cdot \chi_{E_{3}}$ and ${\widetilde{\widetilde{\widetilde{f}}}}_{n}={\widetilde{\widetilde{f}}}_{n}\cdot \chi_{E^{c}_{3}}(n\geq N_{3})$.
Then, by (\ref{56}), we get
$$\sum_{k\in F_{2}}|x_{N_{2}}(k)|\geq \sum_{k=1}^{\infty}\widetilde{\widetilde{\widetilde{f}_{3}}}(k)x_{N_{2}}(k)>\epsilon-\frac{\delta}{2}-\frac{\delta}{2^{2}},$$
where $F_{2}$ is the support of $\widetilde{\widetilde{\widetilde{f}_{3}}}$.

Since $\widetilde{\widetilde{f}_{2}}$ is disjoint from $\widetilde{\widetilde{f}}_{N_{2}}$, the set $F_{2}$ is disjoint from $F_{1}$.

As in the proof of Theorem \ref{5.8}, we get a subsequence $(x_{N_{k}})_{k}$ of $(x_{n})_{n}$ and a sequence $(F_{k})_{k}$ of pairwise disjoint subsets of $\mathbb{N}$ such that
$$\sum_{i\in F_{k}}|x_{N_{k}}(i)|>\epsilon-\sum_{i=1}^{k}\frac{\delta}{2^{i}}>\epsilon-\delta, \quad k=1,2,\cdots$$

Finally, we claim: $wk_{X}(A)\geq \epsilon-\delta.$

Let us fix $n\in \mathbb{N}$. Then, for each $k\in \mathbb{N}$, we have
\begin{align*}
\epsilon-\delta&<\sum_{i=1}^{\infty}|x_{N_{k}}(i)|\chi_{F_{k}}(i)\\
&=\sum_{i=1}^{n}|x_{N_{k}}(i)|\chi_{F_{k}}(i)+\sum_{i=n+1}^{\infty}|x_{N_{k}}(i)|\chi_{F_{k}}(i)\\
&\leq \sum_{i=1}^{n}\chi_{F_{k}}(i)+\sup_{x\in A}\sum_{i=n+1}^{\infty}|x(i)|.\\
\end{align*}
Since $(F_{k})_{k}$ is disjoint pairwise, $\sum_{i=1}^{n}\chi_{F_{k}}(i)=0$ for $k$ large enough.

This implies $$\epsilon-\delta\leq \inf_{n}\sup_{x\in A}\sum_{i=n+1}^{\infty}|x(i)|.$$

It follows from \cite[Proposition 7.3]{KKS} that $$\epsilon-\delta\leq wk_{X}(A).$$
Since $\delta>0$ is arbitrary, we get $$\epsilon\leq wk_{X}(A).$$
The arbitrariness of $\epsilon\in (0,\iota_{1}(A))$ concludes the proof.

\end{proof}

\begin{defn}
Let $1\leq p\leq\infty$. We say that a Banach space $X$ has \textit{quantitative Pe{\l}czy\'{n}ski's property (V) of order $p$} (property $p$-$(V)_{q}$ in short) with a constant $C>0$ if for every Banach space $Y$ and every operator $T: X\rightarrow Y$, one has $$wk_{X^{*}}(T^{*})\leq C\cdot uc_{p}(T).$$ If a Banach space $X$ has property $p$-$(V)_{q}$ with some constant $C>0$, we say that $X$ has property $p$-$(V)_{q}$.
\end{defn}
Clearly, if a Banach space $X$ has property $p$-$(V)_{q}$, then it has property $p$-(V).

Let $1\leq p\leq\infty$ and $X$ be a Banach space. For a bounded subset $A$ of $X^{*}$, we set
$$\eta_{p}(A)=\sup\{\limsup_{n}\sup_{x^{*}\in A}|<x^{*},x_{n}>|: (x_{n})_{n}\in l^{w}_{p}(X), \|(x_{n})_{n}\|_{p}^{w}\leq 1\}.$$

The proof of the following theorem is similar to that of Theorem \ref{5.10}.
\begin{thm}
Let $X$ be a Banach space and $1\leq p<\infty$. The following statements are equivalent:
\item[(1)]$X$ has property $p$-$(V)_{q}$;
\item[(2)]there exists a constant $C>0$ such that for each bounded subset $A$ of $X^{*}$, one has
$$wk_{X^{*}}(A)\leq C\cdot\eta_{p}(A).$$
\end{thm}

\begin{thm}
Let $1\leq p\leq \infty$. If a Banach space $X$ has property $p$-$(V)_{q}$ with a constant $C$, then every quotient of $X$ has property $p$-$(V)_{q}$ with $2C$.
\end{thm}
\begin{proof}
Let $M$ be a closed subspace of $X$. Suppose that $X$ has property $p$-$(V)_{q}$ with a constant $C>0$. We'll show that the quotient $X/M$ has property $p$-$(V)_{q}$ with $2C$. Fix a Banach space $Y$ and an operator $S: X/M \rightarrow Y$. Then we have
$$wk_{X^{*}}(Q^{*}S^{*})\leq C\cdot uc_{p}(SQ),$$ where $Q: X\rightarrow X/M$ is the canonical quotient map.
By (\ref{2}), we get $$wk_{(X/M)^{*}}(S^{*})\leq 2wk_{X^{*}}(Q^{*}S^{*})\leq 2C\cdot uc_{p}(SQ)\leq 2C\cdot uc_{p}(S),$$
which completes the proof.

\end{proof}

\begin{thm}\label{5.3}
Let $X$ be a Banach space and $1\leq p\leq\infty$. Then
\item[(1)]If $X$ has property $p$-$(V)_{q}$ with a constant $C$, then $X^{*}$ has property $p$-$(V^{*})_{q}$ with the same constant $C$;
\item[(2)]If $X^{*}$ has property $p$-$(V)_{q}$ with a constant $C$, then $X$ has property $p$-$(V^{*})_{q}$ with $2C$.
\end{thm}
\begin{proof}
(1)\quad Let $Z$ be a Banach space and $S$ be an operator from $Z$ into $X^{*}$. Applying the assumption to $S^{*}J_{X}$, we get $$wk_{X^{*}}(J_{X}^{*}S^{**})\leq C\cdot uc_{p}(S^{*}J_{X}).$$ This yields
$$wk_{X^{*}}(S)=wk_{X^{*}}(J_{X}^{*}S^{**}J_{Z})\leq wk_{X^{*}}(J_{X}^{*}S^{**})\leq C\cdot uc_{p}(S^{*}J_{X})\leq C\cdot uc_{p}(S^{*}),$$
which completes the proof of (1).

The assertion (2) follows immediately from (1) and Theorem \ref{5.5}.

\end{proof}

Let us mention that the converse of Theorem \ref{5.3} is false. Indeed, $X=(\sum_{n=1}^{\infty}\oplus l^{n}_{\infty})_{1}$ enjoys property $1$-$(V^{*})_{q}$ with $1$ that follows from the following Theorem \ref{5.13}. But $X^{*}$ fails property $p$-(V) for each $1\leq p<\infty$ because $X^{*}$ contains a $1$-complemented subspace isometric to $l_{1}$ and $l_{1}$ fails
property $p$-(V) for each $1\leq p<\infty$ by Theorem \ref{3.8}.

\begin{thm}\label{5.13}
Let $X=(\sum_{\gamma\in \Gamma}\oplus X_{\gamma})_{1}(\Gamma$ an infinite set), where each $X_{\gamma}$ is reflexive. Then
$$wk_{X}(A)\leq\iota_{1}(A)$$ for each bounded subset $A$ of $X$.
\end{thm}
\begin{proof}
Let $A$ be a bounded subset of $X$. We may assume that $wk_{X}(A)>0$ and fix an arbitrary $c\in(0,wk_{X}(A))$. By \cite[Lemma 7.2]{KKS}, we obtain, by induction, a sequence $(F_{n})_{n}$ of pairwise disjoint finite subsets of $\Gamma$ and a sequence $(x_{n})_{n}$ in $A$ such that
$$\sum_{\gamma\in F_{n}}\|x_{n}(\gamma)\|>c,\quad n=1,2,\cdots$$
By Hahn-Banach Theorem, for each $n\in \mathbb{N}$, there exists $(f_{n}(\gamma))_{\gamma\in F_{n}}\in (\sum_{\gamma\in F_{n}}\oplus X^{*}_{\gamma})_{\infty}$ with $\max_{\gamma\in F_{n}}\|f_{n}(\gamma)\|=1$ such that $$\sum_{\gamma\in F_{n}}<f_{n}(\gamma),x_{n}(\gamma)>c.$$
For each $n\in\mathbb{N}$, define $\varphi_{n}\in (\sum_{\gamma\in \Gamma}\oplus X^{*}_{\gamma})_{\infty}$ by

\[\varphi_{n}(\gamma)= \left\{ \begin{array}
                    {r@{\quad,\quad}l}

 f_{n}(\gamma) & \gamma\in F_{n}\\ 0 & otherwise
 \end{array} \right. \]

Since $(F_{n})_{n}$ is pairwise disjoint, it is easy to verify that $(\varphi_{n})_{n}$ is weakly $1$-summable and $\|(\varphi_{n})_{n}\|^{w}_{1}\leq 1$. Moreover, for each $n\in\mathbb{N}$,
$$\sup_{x\in A}|<\varphi_{n},x>|\geq |<\varphi_{n},x_{n}>|=\sum_{\gamma\in F_{n}}<f_{n}(\gamma),x_{n}(\gamma)>c.$$
This yields that $\iota_{1}(A)\geq c$. Since $c\in(0,wk_{X}(A))$ is arbitrary, we get the conclusion.

\end{proof}

\begin{thm}
Let $X=c_{0}(\mathbb{N}, \mathbb{R})$. Then $$wk_{X}(A)=\eta_{1}(A)$$ for each bounded subset $A$ of $X^{*}$.
\end{thm}

The proof is essentially analogous to Theorem \ref{5.12}, only interchanging the role of $X$ and $X^{*}$.

H. Kruli\v{s}ov\'{a} proved in \cite{K} that $C_{0}(\Omega)$ has property $1$-$(V)_{q}$ with constant $\pi$ (constant 2 in the real case) for every locally compact space $\Omega$.

\section*{Acknowledgements} This work is done during Chen's visit to Department of Mathematics, Texas A\&M University. We would like to thank Professor W. B. Johnson for helpful discussions and comments.


\section*{References}
\begin{thebibliography}{99}

\bibitem[1]{A1}
A. Andrew, {\em James' quasi-reflexive space is not isomorphic to any subspace of its dual}, Israel J. Math. {\bf 38}(1981),276-282.


\bibitem[2]{A}
K. T. Andrews, {\em Dunford-Pettis sets in the space of Bochner integrable functions}, Math. Ann. {\bf 241}(1979),35-41.

\bibitem[3]{AC}
C. Angosto and B. Cascales, {\em Measures of weak non-compactness in Bananch spaces}, Topology Appl. {\bf 156}(2009), 1412-1421.

\bibitem[4]{AK}
F. Albiac and N. J. Kalton,
{\em Topics in Banach space theory}, Springer, 2005.


\bibitem[5]{B}
F. Bombal, {\em On Pe{\l}czy\'{n}ski's property $(V^{*})$ in vector sequence spaces}, Collect. Math. {\bf 39}(1988), 141-148.


\bibitem[6]{BD}J. Bourgain and F. Delbaen,
{\em A class of special $\mathcal{L}_{\infty}$ spaces}, Acta Math. {\bf 145}(1980),no.3--4, 155--176.

\bibitem[7]{CCL}
D. Chen, J. Alejandro Ch\'{a}vez-Dom\'{i}nguez and L. Li, {\em Unconditionally $p$-converging operators and Dunford-Pettis Property of order $p$}, to appear.

\bibitem[8]{CE}
R. Cilia and G. Emmanuele, {\em Pe{\l}czy\'{n}ski's property (V) and weak* basic sequences},
Quaestiones Mathematicae. {\bf 38}(2015), 307-316.

\bibitem[9]{CG}
J. M. F. Castillo and M. Gonz\'{a}lez, {\em Properties (V) and (u) are not three-space properties}, Glasgow Math.J. {\bf 36}
(1994), 297-299.

\bibitem[10]{CS}
J. M. F. Castillo and F. S\'{a}nchez, {\em Weakly $p$-compact, $p$-Banach-Saks and super-reflexive Banach spaces},
J. Math. Anal. Appl. {\bf 185}(1994), 256-261.


\bibitem[11]{E1}
G. Emmanuele, {\em A dual characterization of Banach spaces not containing $l^{1}$}, Bull. Acad. Polon. Sci. {\bf 34}(1986), 155-160.


\bibitem[12]{E}
G. Emmanuele, {\em On the Banach spaces with the property ($V^{*}$) of Pe{\l}czy\'{n}ski}, Annali. Mat. Pura ed Applicata. {\bf 152}(1988), 171-181.


\bibitem[13]{FGJ}
T. Figiel, N. Ghoussoub and W. B. Johnson, {\em On the structure of non-weakly compact operators on Banach lattices}, Math. Ann. {\bf 257}(1981), 317-334.


\bibitem[14]{GJ}
N. Ghoussoub and W. B. Johnson, {\em Counterexamples to several problems on the factorization of bounded linear operators},
Proc. Amer. Math. Soc. {\bf 92}(1984), 233-238.


\bibitem[15]{HM}
R. Heinrich and P. Mankiewicz, {\em Applications of ultrapowers to the uniform and Lipschitz classification of Banach spaces}, Studia Math. {\bf 73}(1982), 225-251.


\bibitem[16]{JZ}
W.B.Johnson and M.Zippin, {\em Separable $L_{1}$-preduals are quotients of $C(\Delta)$}, Israel J. Math. {\bf 16}(1973), 198-202.



\bibitem[17]{KKS}
M. Ka\v{c}ena, O. F. K. Kalenda and J. Spurn\'{y}, {\em Quantitative Dunford-Pettis property}, Adv. Math.
{\bf 234}
(2013), 488-527.


\bibitem[18]{KS}
O. F. K. Kalenda and J. Spurn\'{y}, {\em Quantifications of the reciprocal Dunford-Pettis property}, Studia Math. {\bf 210}(2012), 261-278.


\bibitem[19]{K}
H. Kruli\v{s}ov\'{a}, {\em Quantification of Pe{\l}czy\'{n}ski's property (V)}, to appear.



\bibitem[20]{LT} J. Lindenstrauss and L. Tzafriri,
{\em Classical Banach Spaces I, Sequence Spaces}, Springer,
Berlin, 1977.

\bibitem[21]{P}
A. Pe{\l}czy\'{n}ski, {\em On Banach spaces on which every unconditionally converging operator is weakly compact}, Bull. Acad. Polon. Sci. {\bf 10}(1962), 641-648.

\end{thebibliography}
\end{document}